\documentclass[12pt]{amsart}
\title{Generalizations of Fourier analysis, and how to apply them}
\author{W.T. Gowers}

\usepackage{amsmath, wrapfig}
\usepackage{dsfont, a4wide, amsthm, amssymb, amsfonts, graphicx}
\usepackage{fancyhdr, xspace, psfrag, setspace, supertabular, color}
\usepackage{caption}
\usepackage{hyperref}

\begin{document}

\newtheorem{theorem}{Theorem}[section]
\newtheorem{proposition}[theorem]{Proposition}
\newtheorem{lemma}[theorem]{Lemma}
\newtheorem*{claim}{Claim}
\newtheorem{corollary}[theorem]{Corollary}
\newtheorem{conjecture}[theorem]{Conjecture}
\newtheorem{definition}[theorem]{Definition}
\newtheorem{problem}[theorem]{Problem}
\newtheorem{example}[theorem]{Example}
\newtheorem{question}[theorem]{Question}
\newtheorem{remark}[theorem]{Remark}

\def\e{\epsilon}
\def\a{\alpha}
\def\b{\beta}
\def\g{\gamma}
\def\d{\delta}
\def\cf{\mathcal{F}}
\def\E{\mathbb{E}}
\def\N{\mathbb{N}}
\def\R{\mathbb{R}}
\def\C{\mathbb{C}}
\def\Z{\mathbb{Z}}
\def\F{\mathbb{F}}
\def\hf{\hat f}
\def\T{\mathbb{T}}
\def\tr{\mathrm{tr}}

\onehalfspacing

\maketitle

\begin{abstract}
This is a survey of the use of Fourier analysis in additive combinatorics, with a particular focus on situations where it cannot be straightforwardly applied, but needs to be generalized first. Sometimes very satisfactory generalizations exist, while sometimes we have to make do with theories that have some of the desirable properties of Fourier analysis but not all of them. In the latter case, there are intriguing hints that there may be more satisfactory theories yet to be discovered. This article grew out of the Colloquium Lectures at the Joint Meeting of the AMS and the MAA, given in Seattle in January 2016.
\end{abstract}

\section{Introduction: What is additive combinatorics?}

Additive combinatorics is a newish and very active branch of mathematics that grew out of combinatorial number theory, with input from many other areas such as harmonic analysis, ergodic theory, analytic number theory, group theory, and extremal combinatorics. It has since fed back into those areas and led to the solutions of several long-standing open problems. Because of all these connections and influences, the subject is not very easy to characterize, but a good way to understand the flavour of the area is to look at one of its central theorems, the following famous result of Szemer\'edi from 1974 \cite{szem75}, which solved a conjecture made by Erd\H os and Tur\'an in 1936.

\begin{theorem} \label{szem} 
For every positive integer $k$ and every $\d>0$ there exists a positive integer $n$ such that every subset $A\subset\{1,2,\dots,n\}$ of size at least $\d n$ contains an arithmetic progression of length $k$.
\end{theorem}

This is a combinatorial theorem in the sense that we make no structural assumptions about $A$ -- it is just a subset of $\{1,2,\dots,n\}$ of density at least $\d$. However, the set $\{1,2,\dots,n\}$ has a rich additive structure, and that structure is highly relevant to the problem, since an arithmetic progression can be thought of as a sequence $(x_1,x_2,\dots,x_k)$ such that 
\[x_2-x_1=x_3-x_3=\dots=x_k-x_{k-1}.\]
(Of course, we also need to add the non-degeneracy condition that $x_1\ne x_2$.) 

However, there is more to additive combinatorics than a set of combinatorial theorems that involve addition in one way or another. To appreciate this, it is helpful to look at the following statement, which turns out to be an equivalent reformulation of Szemer\'edi's theorem. The equivalence is a reasonably straightforward exercise to prove.

\begin{theorem} For every positive integer $k$ and every $\d>0$ there exists a constant $c>0$ such that for every positive integer $n$ and every function $f:\Z_n\to[0,1]$ that averages at least $\d$ we have the inequality
\[\E_{x,d}f(x)f(x+d)\dots f(x+(k-1)d)\geq c.\]
\end{theorem}

Here $\Z_n$ is the cyclic group of order $n$ and the notation $\E_{x,d}$ means the average over all $x$ and $d$ -- that is, it is another way of writing $n^{-2}\sum_{x,d}$. 

As $n$ gets large, $\Z_n$ is a better and better discrete approximation to the circle $\T$, which we can think of as the Abelian group consisting of all complex numbers of modulus 1. It is not hard to prove that the discrete statement is equivalent to the following continuous version.

\begin{theorem} For every positive integer $k$ and every $\d>0$ there exists a constant $c>0$ such that for every measurable function $f:\T\to[0,1]$ that averages at least $\d$ we have the inequality
\[\E_{x,d} f(x)f(x+d)\dots f(x+(k-1)d)\geq c.\]
\end{theorem}
\noindent This time $\E_{x,d}$ stands for the integral with respect to the Haar measure on $\T^2$. 

This last reformulation illustrates an important point about many of the theorems of additive combinatorics (and extremal combinatorics more generally), which is that although they are combinatorial, they are also analytic. In fact, the more one thinks about them, the less important the distinction between discrete and continuous seems to be. And it is not just the statements that are (or can be made to be) analytic: a characteristic feature of much of additive combinatorics is that the proofs of its theorems use methods from areas of analysis such as functional analysis, Fourier analysis, and ergodic theory. 

Here we shall focus on the second of these. Fourier analysis is an extremely useful tool for additive problems, and one of the aims of this survey will be to explain why. Another aim, which is in some ways even more interesting, will be to demonstrate the \emph{limitations} of Fourier analysis -- that is, to look at problems that do not immediately yield to a Fourier-analytic approach. Sometimes that just means that one needs to look for a completely different kind of argument. However, with some problems the best way to make progress is not to abandon Fourier analysis altogether, but to generalize it in a suitable, and not always obvious, way. Thus, it sometimes happens that the limitations of one type of Fourier analysis lead to the development of another.

\section{Discrete Fourier analysis}

Let $f:\Z_n\to\C$. We define its \emph{discrete Fourier transform} $\hat f:\Z_n\to\C$ by the formula
\[\hat f(r)=\E_xf(x)\omega^{-rx},\]
where $\omega=\exp(2\pi ix/n)$ is a primitive $n$th root of unity. Note that there is a close resemblance between this formula, which we could equally well write as
\[\hat f(r)=\E_x f(x)\exp(-2\pi i rx/n),\]
and the familiar formulae for Fourier coefficients and Fourier transforms in the continuous setting. Of course, this is to be expected. Note also that the number $\omega^{-rx}$ is well-defined, since if $r$ and $n$ are integers, then adding a multiple of $n$ to either of them makes no difference to it.

Although $\hat f$ can be thought of as a function defined on $\Z_n$, it is more correct to regard it as defined on the dual group $\hat\Z_n$, which happens to be (non-naturally) isomorphic to $\Z_n$. The distinction has some importance in additive combinatorics, because the natural measures we put on $\Z_n$ and $\hat\Z_n$ are different: for $\Z_n$ we use the uniform probability measure, whereas for $\hat\Z_n$ we use the counting measure. This difference feeds into the definitions of some key concepts such as inner products, $p$-norms and convolutions. Given functions $f,g:\Z_n\to\C$ and $1\leq p\leq\infty$, we have the following definitions.
\begin{itemize}
\item $\langle f,g\rangle = \E_xf(x)\overline{g(x)}$.
\item $\|f\|_p=(\E_x|f(x)|^p)^{1/p}$.
\item $f*g(x)=\E_{y+z=x}f(y)g(z)$.
\end{itemize}
The corresponding definitions for functions $\hat f,\hat g:\hat\Z_n\to\C$ are the same, but with sums replacing averages. That is, they are as follows.
\begin{itemize}
\item $\langle \hat f,\hat g\rangle = \sum_x\hat f(x)\overline{\hat g(x)}$.
\item $\|\hat f\|_p=(\sum_x|\hat f(x)|^p)^{1/p}$.
\item $\hat f*\hat g(x)=\sum_{y+z=x}\hat f(y)\hat g(z)$.
\end{itemize}
With these measures in place, the familiar properties of the Fourier transform hold for the discrete Fourier transform as well, and have easier proofs. In particular, constant use is made of the following five rules, of which the first two are equivalent. All five are easy exercises.
\begin{itemize}
\item $\langle f,g\rangle = \langle \hat f,\hat g\rangle$ (Parseval's identity).
\item $\|f\|_2=\|\hat f\|_2$ (also Parseval's identity).
\item $f(x)=\sum_r\hat f(r)\omega^{rx}$ (the inversion formula).
\item $\widehat{f*g}(r)=\hat f(r)\hat g(r)$ (the convolution identity).
\item If $a$ is invertible mod $n$ and $g(x)=f(ax)$ for every $x\in\Z_n$, then $\hat g(r)=\hat f(a^{-1}r)$ for every $r$ (the dilation rule).
\end{itemize}
In additive combinatorics, one often deals with characteristic functions of subsets $A$ of $\Z_n$, and some authors like to use the letter $A$ for its own characteristic function: that is, $A(x)=1$ if $x\in A$ and 0 otherwise. Given a subset $A\subset\Z_n$, define its \emph{density} to be $|A|/n$. The following three observations are often used.
\begin{itemize}
\item $\hat A(0)=\a$.
\item $\sum_r|\hat A(r)|^2=\a$.
\item $\hat A(-r)=\overline{\hat A(r)}$.
\end{itemize}
The first observation is immediate from the definition, the second follows from Parseval's identity and the fact that $\sum_r|\hat A(r)|^2=\|\hat A\|_2^2$, and the third follows from the fact that $A$ is real-valued and that $\omega^{rx}=\overline{\omega^{-rx}}$ for every $r$ and $x$ (and so is true of all real-valued functions). 

\section{Roth's theorem}

To give an idea of how useful these simple facts are, we shall now sketch a proof of Roth's theorem, which is the case $k=3$ of Szemer\'edi's theorem (Theorem \ref{szem} of these notes). Thus, we would like to prove the following theorem. It was proved by Roth in 1953 \cite{roth53}.

\begin{theorem}
For every $\d>0$ there exists $n$ such that every subset $A\subset\{1,2,\dots,n\}$ of density at least $\d$ contains an arithmetic progression of length 3.
\end{theorem}

In order to apply Fourier analysis, it is convenient to think of $A$ as a subset of $\Z_n$ rather than of $\{1,2,\dots,n\}$. (This is not essential, however: Roth originally treated $A$ as a subset of $\Z$.) We shall also assume that $n$ is odd. Let us write $A_2$ for the function defined by $A_2(z)=A(z/2)$, which is the characteristic function of the set of $z$ such that $z/2\in A$. Because $n$ is odd, the map $z\mapsto z/2$ is a well-defined bijection. 

The key observation that shows why Fourier analysis is useful is that the number of arithmetic progressions in $A$ can be expressed in terms of convolutions, inner products and dilations, and therefore has a neat expression in terms of the Fourier coefficients of $A$. Indeed, using the rules given earlier, we have that
\begin{eqnarray*}
\E_{x+y=2z}A(x)A(y)A(z)&=&\E_{x+y=z}A(x)A(y)A(z/2)\\
&=&\E_z A*A(z)A_2(z)\\
&=&\langle A*A,A_2\rangle\\
&=&\langle\widehat{A*A},\hat A_2\rangle \\
&=&\langle\hat A^2,\hat A_2\rangle \\
&=&\sum_r\hat A(r)^2\overline{\hat A_2(r)}\\
&=&\sum_r\hat A(r)^2\overline{\hat A(2r)} \\
&=&\sum_r\hat A(r)^2\hat A(-2r).\\
\end{eqnarray*}

Why should this be useful? To answer that question, we need to bring in another simple but surprisingly powerful tool: the Cauchy-Schwarz inequality. First, recalling that $\hat A(0)$ is equal to the density of $A$, which we shall denote by $\a$, we split the last expression up as
\[\a^3+\sum_{r\ne 0}\hat A(r)^2\hat A(-2r).\]
Thus, we have shown that
\[\E_{x+y=2z}A(x)A(y)A(z)=\a^3+\sum_{r\ne 0}\hat A(r)^2\hat A(-2r).\]
The left-hand side of this expression is the probability that $x,y,z$ all belong to $A$ if you choose them randomly to satisfy the equation $x+y=2z$. Without the constraint that $x+y=2z$ this probability would be $\a^3$, since each of $x,y$ and $z$ would have a probability $\a$ of belonging to $A$. So the term $\a^3$ on the right-hand side can be thought of as ``what one would expect" and the remainder of the right-hand side is a measure of the effect of the dependence of $x,y$ and $z$ on each other. 

However, this effect depends significantly on $A$. If the elements of $A$ are chosen independently at random with probability $\a$, then for each pair of distinct $x,y$ the events $x\in A$, $y\in A$ and $(x+y)/2\in A$ are independent, so restricting the average to triples $(x,y,z)$ such that $x+y=2z$ will typically have very little effect. By contrast, if $A$ is an interval of length $\a n$, then the events become highly correlated. So the term $\sum_{r\ne 0}\hat A(r)^2\hat A(-2r)$ is a measure of quasirandomness of $A$: the smaller it is, the less the events $x\in A$, $y\in A$ and $z\in A$ are correlated if $x,y,z$ are chosen randomly to satisfy the constraint $x+y=2z$.

It is to bound the remainder term that we use the Cauchy-Schwarz inequality, and also the even more elementary inequality $|\langle f,g\rangle|\leq\|f\|_1\|g\|_\infty$. We find that
\begin{eqnarray*}|\sum_{r\ne 0}\hat A(r)^2\hat A(-2r)|&\leq&\max_{r\ne 0}|\hat A(r)|\sum_{r\ne 0}|\hat A(r)||\hat A(-2r)|\\
&\leq&\max_{r\ne 0}|\hat A(r)|(\sum_r|\hat A(r)|^2)^{1/2}(\sum_r|\hat |A(-2r)|^2)^{1/2}\\
&=&\max_{r\ne 0}|\hat A(r)|\|\hat A\|_2^2\\
&=&\a\max_{r\ne 0}|\hat A(r)|.\\
\end{eqnarray*}
It follows that 
\[\E_{x+y=2z}A(x)A(y)A(z)\geq\a^3-\a\max_{r\ne 0}|\hat A(r)|.\]
We see from this that if all the Fourier coefficients $\hat A(r)$ are small (more precisely, if they all have size significantly less than $\a^2$), then the number of triples $(x,y,z)\in A^3$ with $x+y=2z$ is indeed close to $\a^3n^2$, the approximate number we would get if the elements of $A$ were chosen independently at random, each with probability $\a$. 

Therefore, either we have the arithmetic progression we are looking for (strictly speaking, this is incorrect because our triples satisfy the equation $x+y=2z$ in $\Z_n$ and not necessarily in $\Z$ when we regard $x,y$ and $z$ as ordinary integers, but this is a technical problem that can be dealt with), or $A$ has a large Fourier coefficient $\hat A(r)$ for some non-zero $r$. Here, ``large" can be taken to mean ``of absolute value at least $c\a^2$" for some absolute constant $c>0$.

In the second case, let us define a function $f:\Z_n\to\R$ by setting $f(x)=A(x)-\a$ for each $x$. It is easy to show that $\hat f(r)=\hat A(r)$ (this uses the fact that $r\ne 0$). So we obtain an inequality 
\[|\hat f(r)|=|\E_xf(x)\omega^{-rx}|\geq c\a^2.\]
At this point we use a lemma, which I shall state imprecisely.

\begin{lemma}
For every $r\ne 0$ there exists a partition of $\Z_n$ into arithmetic progressions $P_1,\dots,P_m$, each of length at least $c\sqrt n$, such that the function $\omega^{rx}$ is approximately constant on each $P_i$.
\end{lemma}

The proof of the lemma is an exercise based on a well-known technique: one uses the fact that by the pigeonhole principle it is possible to find $0\leq u<v$ such that $v$ is not too large and $|\omega^{ru}-\omega^{rv}|=|1-\omega^{r(v-u)}|$ is small. One can then partition $\Z_n$ into arithmetic progressions of common difference $v-u$.

Given the lemma, one observes that
\[c\a^2 n\leq|\sum xf(x)\omega^{-rx}|\leq\sum_i|\sum_{x\in P_i}f(x)\omega^{-rx}|\approx\sum_i|\sum_{x\in P_i}f(x)|,\]
and also that
\[0=\sum_xf(x)=\sum_i\sum_{x\in P_i}f(x).\]
Adding these equations together and using an averaging argument, we find that there exists $i$ such that
\[|\sum_{x\in P_i}f(x)|+\sum_{x\in P_i}f(x)\geq c'\a^2|P_i|,\]
where $c'$ is a slightly smaller absolute constant (because of the approximation in the first equation), which implies that
\[\sum_{x\in P_i}f(x)\geq c'\a^2|P_i|.\]
Recalling that $f(x)=A(x)-\a$ for each $x$, we find that this is telling us that 
\[|A\cap P_i|\geq(\a+c'\a^2)|P_i|.\]
Thus, what we have managed to do is find an arithmetic progression $P_i$ of length at least $c\sqrt n$ such that the density of $A$ inside $P_i$ is greater than the density of $A$ inside $\Z_n$ by $c'\a^2$. 

We can iterate this argument: either $A\cap P_i$ contains an arithmetic progression of length $3$ or $P_i$ contains a subprogression of length at least $c\sqrt{|P_i|}$ inside which $A$ has density at least $\a+2c'\a^2$, and so on. The iteration must eventually terminate, because the density cannot exceed 1, and Roth's theorem is proved. 

If one analyses carefully the bound that comes out of the above argument, one finds that it shows that if $A$ is a subset of $\{1,2,\dots,n\}$ of density at least $C/\log\log n$, for some absolute constant $C$, then $A$ must contain an arithmetic progression of length 3. The double logarithm comes from the fact that we have to iterate $\a^{-1}$ times and each time we do so we take a square root.

This bound has been improved in interesting ways several times. The following table gives an idea of how the bounds have progressed over the years. The publication dates of the papers of Szemer\'edi and Heath Brown are slightly misleading: those results were actually independent. Also, the papers of Sanders obviously came out in the opposite order to the order in which the results were proved.
\begin{table}[ht]
\caption*{Bounds for Roth's theorem}
\centering
\begin{tabular} {c c c c}
\hline\hline
Author & Density bound & Published & Reference \\
\hline
Roth & $C/\log\log n$ & 1953 & \cite{roth53} \\
Heath-Brown & $C/(\log n)^c$, some $c>0$ & 1987 & \cite{heathbrown87} \\
Szemer\'edi & $C/(\log n)^{1/20}$ & 1990 & \cite{szem90} \\
Bourgain & $C(\log\log n/\log n)^{1/2}$ & 1999 & \cite{bourgain99} \\
Bourgain & $C(\log\log n)^2/(\log n)^{2/3}$ & 2008 & \cite{bourgain08} \\
Sanders & $(\log n)^{-3/4 + o(1)}$ & 2012 & \cite{sanders12:1} \\
Sanders & $C(\log\log n)^6/\log n$ & 2011 & \cite{sanders11} \\
Bloom & $C(\log\log n)^4/\log n$ & 2012 & \cite{bloom16} \\
\hline
\end{tabular}
\end{table}
The problem of improving the bounds for Roth's theorem has been an extremely fruitful one: the 2008 paper of Bourgain and the 2012 paper of Sanders could perhaps be regarded as clever refinements of existing techniques, but all the other papers introduced significant new ideas, many of which have been very influential and led to the solutions of several other problems.

To put these results in perspective, it is worth mentioning that the best known lower bound on the density (that is, the largest density known to be possible for a set that contains no progression of length 3) is $\exp(-c\sqrt{\log n})$, which is far lower than Bloom's current record upper bound. But even if that gap turns out to be very hard to close, we are tantalizingly close to a bound of $1/\log n$, which would be enough to give a purely combinatorial proof that the primes contain infinitely many arithmetic progressions of length 3 (a result that was proved by number-theoretic methods soon after Vinogradov proved his 3-primes theorem). In fact, a bound of $c\log\log n/\log n$ would suffice for this, since the fact that the primes have very small intersection with some arithmetic progressions (such as the even numbers) can be used to show that there are arithmetic progressions of length $n$ inside which the primes have at least that density.

\section{A first generalization -- to arbitrary finite Abelian groups}

Many of the proof techniques that give us results about subsets of $\Z_n$ work just as well in an arbitrary Abelian group. This turns out to be a very useful observation, as there are some Abelian groups, in particular the groups $\F_p^n$ for fixed $p$ and large $n$, where the proofs are much cleaner. So sometimes to work out the proof of a result about $\Z_n$ it is a good strategy to prove an analogue for a group such as $\F_3^n$ first and then work out how to modify the argument so that it works in $\Z_n$. (For a much fuller explanation of the benefits of this strategy, a survey by Ben Green from 2005 \cite{green05} and a follow-up by Julia Wolf written a decade later \cite{wolf15} are highly recommended.)

Recall the inversion formula for the Fourier transform on $\Z_n$, which states that
\[f(x)=\sum_r\hat f(r)\omega^{rx}.\]
If we write $\omega_r$ for the function $x\mapsto\omega^{rx}$, then we can write the formula in the slightly more abstract form
\[f=\sum_r\hat f(r)\omega_r,\]
which is showing us how to write $f$ as a linear combination of the functions $\omega_r$.

What is special about the functions $\omega_r$? The property that singles them out is that they are the \emph{characters} of $\Z_n$, that is, the homomorphisms from $\Z_n$ to $\C$. It turns out to be straightforward to generalize Fourier analysis to all finite Abelian groups $G$ by decomposing functions $f:G\to\C$ as linear combinations of characters.

For this to work, we would like the characters to form an orthonormal basis, which they do, by a well-known argument. To see the orthonormality, let $\chi$ be a non-trivial character, let $y\in G$ be such that $\chi(y)\ne 1$, and observe that
\[\E_x\chi(x)=\E_x\chi(xy)=\chi(y)\E_x\chi(x),\]
from which it follows that $\E_x\chi(x)=0$. But then if $\chi_1$ and $\chi_2$ are distinct characters, we have that
\[\langle\chi_1,\chi_2\rangle=\E_x\chi_1(x)\overline{\chi_2(x)}=\E_x\chi_1(x)\chi_2(x)^{-1},\]
which is zero, since $\chi_1\chi_2^{-1}$ is a non-trivial character.

Less elementary is the fact that the characters span $G$. For this one needs the structure theorem for finite Abelian groups, which gives us that $G$ is a product of cyclic groups. We know that each cyclic group has a complete basis of characters, and the products of those characters form a basis of characters for the whole group, which gives us a complete set.

Given that the characters form an orthonormal basis, we can expand a function $f$ as a linear combination $\sum_\chi \langle f,\chi\rangle\chi$. The coefficients $\langle f,\chi\rangle$ are called the Fourier coefficients of $f$ and denoted $\hat f(\chi)$. That is, we have the formula
\[\hat f(\chi)=\E_x f(x)\overline{\chi(x)}\]
for the Fourier transform, and the statement that $f=\sum_\chi\langle f,\chi\rangle\chi$ is giving us our inversion formula
\[f(x)=\sum_\chi\hat f(\chi)\chi(x).\]
The fact that we are writing $\hat f(\chi)$ represents a slight change of notation from the $\Z_n$ case, where we wrote $\hat f(r)$ instead of $\hat f(\omega_r)$. This emphasizes the fact that properly speaking the Fourier transform is defined on the dual group $\hat G$ rather than on $G$. It happens that these two groups are isomorphic, but the isomorphism is not natural in the category-theoretic sense, and, as commented earlier, we like to put different measures on them.

When $G$ is the group $\F_3^n$, the characters take the form $\omega_r:x\mapsto \omega^{r.x}$, where now $r$ and $x$ are elements of $\F_3^n$, $\omega=\exp(2\pi i/3)$, and $r.x$ is shorthand for $\sum_{i=1}^nr_ix_i$. It was observed by Meshulam \cite{meshulam95} that Roth's proof of Roth's theorem has an analogue for subsets of $\F_3^n$, and that the proof is in fact considerably simpler in that context because there is no longer any need for the lemma about partitioning into arithmetic progressions on which a character is roughly constant. The theorem is as follows.

\begin{theorem}
There is a constant $C$ such that for every positive integer $n$, every subset $A\subset\F_3^n$ of density at least $C/n$ contains distinct elements $x,y,z$ such that $x+y+z=0$.
\end{theorem}

Note that in $\F_3^n$ the equation $x+y+z=0$ is equivalent to the equation $x+y=2z$, so the analogy with Roth's theorem is very close. As for the proof, one gets in exactly the same way that either $A$ looks random enough that it must contain an arithmetic progression or there is a non-zero $r$ such that $\hat{A}(r)$ (or $\hat A(\omega_r)$ if you prefer) has magnitude at least $c\a^2$, where $\a$ is the density of $A$. In the second case, it is easy to show that $A$ has density at least $\a+c'\a^2$ in at least one of the three sets $\{x:r.x=i\}$ (where $i=0,1$ or $2$). Since these sets are just subspaces of $\F_3^n$ of codimension 1, we are then already in a position to iterate.

This argument illustrates very well why it can be fruitful to look at more general Abelian groups. Because the group $\F_3^n$ has a rich set of cosets of subgroups -- namely all the affine subspaces -- it is very convenient for iterative arguments. This somehow allows one to focus on the ``real issues". In more general Abelian groups, and in particular with the cyclic groups $\Z_n$, one has to make do with subsets that are ``subgroup-like". Doing so is possible, but it creates technical problems that can make arguments hard work to write down and even harder to read. 

Until fairly recently, the best known upper bound was given by the simple argument outlined above. But in 2011 Bateman and Katz improved the bound to one of the form $C/n^{1+\e}$ for fixed constants $C$ and $\e>0$. This was a remarkable achievement, given how long the bound had stood still, but the gap that remained was still huge. 

In the other direction, the best known method for producing lower bounds was to look for an example $B\subset\F_3^k$ for some small $k$ and then to use it to create a class of examples $A=B^r\subset\F_3^{kr}$. If $B$ has density $c^k$ and $n=kr$, then $B^r$ has density $c^{kr}=c^n$. But the following question was left wide open by the result of Bateman and Katz.

\begin{question}
Let $c_n$ be the greatest possible density of a subset $A\subset\F_3^n$ that contains no three distinct elements $x,y,z$ such that $x+y+z=0$. Does there exist $\theta<1$ such that $c_n\leq\theta^n$ for every $n$?
\end{question}

\noindent Very recently -- in May 2016 -- this problem was solved using a completely different method, in a development that astonished additive combinatorialists. First, Croot, Lev and Pach obtained an upper bound of this exponential type for subsets of $\Z_4^n$ with no 3-term arithmetic progression \cite{crootlevpach16}. Then, barely a week later, Jordan Ellenberg and Dion Gijswijt independently saw how to modify the argument of Croot, Lev and Pach to give a similar bound for the cap-set problem itself, thereby giving a positive answer to the question above~\cite{ellenberggijswijt16}. 

The paper of Ellenberg and Gijswijt is short and self-contained, but here is a very brief outline of how the proof goes. Suppose that $A\subset\F_3^n$ is a set that contains no solution to the equation $x+y=2z$. For each $d$, let $Q_d$ be the vector space of polynomials over $\F_3$ in $n$ variables that have degree at most 2 in each variable and that have total degree at most $d$. The polynomials of degree at most 2 in each variable are distinct not just as polynomials but also as functions on $\F_3^n$, so $Q_{2n}$ is the space of all functions from $\F_3^n$ to $\F_3$.

A random polynomial of degree at most 2 in each variable has expected degree $n$, and the probability that its degree deviates significantly from $n$ is tiny. In particular, the probability that a polynomial belongs to $Q_{2n/3}$ or fails to belong to $Q_{4n/3}$ is exponentially small. 

Write $2.A$ for the set $\{2x:x\in A\}$. Then our hypothesis about $A$ can be expressed as the statement that $(A+A)\cap 2.A=\emptyset$. If $A$, and therefore $2.A$, has density significantly larger than the probability that a polynomial fails to belong to $Q_{4n/3}$, then a simple dimension argument shows that there exists a polynomial of degree at most $4n/3$ that vanishes outside $2.A$ and does not vanish inside $2.A$. With a little care, one can show that if $A$ is a bit bigger than this, then there exists a polynomial of degree at most $4n/3$ that is zero outside $2.A$ and non-zero on at least two thirds of the points in $2.A$.

In particular, there is a polynomial $P$ of degree at most $4n/3$ that vanishes on $A+A$ and is non-zero at at least two thirds of the points in $2.A$. This implies that if we define a function $f:A\times A\to\F_3$ by $f(x,y)=P(x+y)$, then this function, considered as a matrix, has rank at least $2|A|/3$, since it is zero off the diagonal and non-zero in at least $2|A|/3$ places on the diagonal.

However, one can also show that $f$ has smaller rank than this, using the fact that it is of the special form $P(x+y)$, where $P$ is a polynomial of degree at most $4n/3$. The idea is to expand out $P(x+y)$ as a linear combination of monomials in the $x_i$ and $y_i$ and divide up the sum according to whether the contribution from the $x_i$ has degree at most $2n/3$ or the contribution from the $y_i$ does. It is straightforward to show then that the rank of $f$ is at most twice the dimension of $Q_{2n/3}$, which, as we have commented, is exponentially small compared with $3^n$.

This proof, though simple, still needs to be fully digested. Is it going to lead to solutions to many other problems, such as the problem of finding the right bounds for Roth's theorem? (It has not been used for that, but it has been used for some other problems already \cite{foxlovasz16, green16}.) Is Fourier analysis about to be dethroned from its position as the tool of choice for this kind of problem? How are the two approaches related, if at all? It is too early to say, but it seems highly likely that there will be further developments in the not too distant future.

\section{The $U^2$ norm}

In the proof of Roth's theorem, we had a useful measure of the quasirandomness of a function, namely the size of its largest Fourier coefficient -- the smaller that size, the more quasirandom the function. However, this measure has the disadvantage that there isn't an obvious physical-space interpretation of $\|\hat f\|_\infty$ -- that is, an expression in terms of the values of $f$ that does not mention the Fourier transform. Instead, one often prefers to use the measure $\|\hat f\|_4$, which does turn out to have a physical-space interpretation. In the contexts we care about, these two quantities are roughly equivalent, since we have the trivial inequalities
\[\|\hat f\|_\infty^4\leq\|\hat f\|_4^4\leq\|\hat f\|_\infty^2\|\hat f\|_2^2,\]
and we usually deal with functions $f$ such that $\|\hat f\|_2=\|f\|_2\leq 1$. This tells us that $\|\hat f\|_\infty$ is small if and only if $\|\hat f\|_4$ is small (though if we pass from one equivalent statement to the other and back again, we obtain a worse constant of smallness than the one we started with).

The reason that $\|\hat f\|_4$ is nice is that
\[\|\hat f\|_4^4=\sum_r|\hat f(r)|^4=\langle\hat f^2,\hat f^2\rangle=\langle f*f,f*f\rangle=\E_{x+y=z+w}f(x)f(y)\overline{f(z)f(w)},\]
where in the above argument we used the definition of the $\ell_4$ norm, the definition of the inner product on $\hat Z_n$, Parseval's identity and the convolution identity, and the definition of convolutions and inner products in $\Z_n$. (It is also possible to prove the identity above using a direct calculation, but it is nicer to use the basic properties of the Fourier transform.) 

Quadruples $(x,y,z,w)$ with $x+y=z+w$ are the same as quadruples of the form $(x,x+a+b,x+a,x+b)$, so the final expression above can be written in the form
\[\E_{x,a,b}f(x)\overline{f(x+a)f(x+b)}f(x+a+b).\]
Since this equals $\|\hat f\|_4^4$, we find that it is possible to define a norm $\|f\|_{U^2}$ by the formula
\[\|f\|_{U^2}=(\E_{x,a,b}f(x)\overline{f(x+a)f(x+b)}f(x+a+b))^{1/4}.\]
This may seem pointless, since it is just renaming the norm $f\mapsto\|\hat f\|_4$, but we use a different name to emphasize that we are using a purely physical-space definition. The great advantage of doing this is that it gives us an alternative definition that has, as we shall see later, a very natural and useful generalization that does not correspond to any direct generalization of the definition in terms of Fourier coefficients. 

A useful fact about the $U^2$ norm is that it satisfies a kind of Cauchy-Schwarz inequality. Let us define a generalized inner product by the formula
\[[f_1,f_2,f_3,f_4]=\E_{x,a,b}f_1(x)\overline{f_2(x+a)f_3(x+b)}f_4(x+a+b).\]
Then $\|f\|_{U^2}^4=[f,f,f,f]$. The inequality states that
\[[f_1,f_2,f_3,f_4]\leq\|f_1\|_{U^2}\|f_2\|_{U^2}\|f_3\|_{U^2}\|f_4\|_{U^2}.\]
We quickly sketch a proof. We have that
\begin{align*}
[f_1,f_2,f_3,f_4]&=\E_{x,y,a}f_1(x)\overline{f_2(x+a)f_3(y)}f_4(y+a)\\
&=\E_a(\E_xf_1(x)\overline{f_2(x+a)})\overline{(\E_yf_3(y)\overline{f_4(y+a)})}\\
&\leq(\E_a|\E_xf_1(x)\overline{f_2(x+a)}|^2)^{1/2}(\E_a|\E_yf_3(y)\overline{f_4(y+a)}|^2)^{1/2}\\
\end{align*}
by the usual Cauchy-Schwarz inequality. But this last expression is easily seen to be
\[[f_1,f_2,f_1,f_2]^{1/2}[f_3,f_4,f_3,f_4]^{1/2}.\]
Furthermore, we have the symmetry $[f_1,f_2,f_3,f_4]=[f_1,f_3,f_2,f_4]$, so we can rewrite the last expression as 
\[[f_1,f_1,f_2,f_2]^{1/2}[f_3,f_3,f_4,f_4]^{1/2}.\]
Applying the argument again we find that 
\[[f_1,f_1,f_2,f_2]\leq[f_1,f_1,f_1,f_1]^{1/2}[f_2,f_2,f_2,f_2]^{1/2},\]
and similarly for $f_3$ and $f_4$, and from this the result follows.

This inequality gives us a generalized Minkowski inequality in just the way that the normal Cauchy-Schwarz inequality gives the normal Minkowski inequality. Indeed,
\begin{align*}\|f_0+f_1\|_{U^2}^4&=[f_0+f_1,f_0+f_1,f_0+f_1,f_0+f_1]\\
&=\sum_{\e\in\{0,1\}^4}[f_{\e_1},f_{\e_2},f_{\e_3},f_{\e_4}]\\
&\leq\sum_{\e\in\{0,1\}^4}\|f_{\e_1}\|_{U^2}\|f_{\e_2}\|_{U^2}\|f_{\e_3}\|_{U^2}\|f_{\e_4}\|_{U^2}\\
&=(\|f_0\|_{U^2}+\|f_1\|_{U^2})^4.\\
\end{align*}
We thus have a proof, entirely in physical space, that the $U^2$ norm is a norm.

If $A\subset\Z_n$, then we can measure the quasirandomness of $A$ as follows. Let $\a$ be the density of $A$ and write $A(x)=\a+f(x)$. Then by the loose equivalence of the $\ell_\infty$ and $\ell_4$ norms of the Fourier coefficients, we have that $A$ is quasirandom in a useful sense if $\|f\|_{U^2}$ is small. One can check easily that $\|A\|_{U^2}^4=\a^4+\|f\|_{U^2}^4$, so this is saying that $\|A\|_{U^2}$ is approximately equal to $\a^4$. But $\|A\|_{U^2}^4$ has a nice interpretation. Recall that it equals
\[\E_{x+y=z+w}A(x)A(y)A(z)A(w),\]
which is the probability, if you choose a random quadruple $(x,y,z,w)$ such that $x+y=z+w$, that all of $x,y,z$ and $w$ lie in $A$. This we call the \emph{additive quadruple density} of $A$. Thus, a set of density $\a$ has additive quadruple density at least $\a^4$, with near equality if it is quasirandom in a useful sense.

An important final remark is that one can also prove entirely in physical space that if $A$ is quasirandom in this sense, then its arithmetic-progression density is roughly $\a^3$. Indeed, writing $A(x)=\a+f(x)$ again, and noting that if we pick a random triple $(x,y,z)$ with $x+y=2z$, then any two of $x,y$ and $z$ will be independent and uniformly distributed (always assuming that $n$ is odd), we have that
\begin{align*}
\E_{x+y=2z}A(x)A(y)A(z)&=\E_{x+y=2z}(\a+f(x))(\a+f(y))(\a+f(z))\\
&=\a^3+\E_{x+y=2z}f(x)f(y)f(z)\\
&=\langle f*f,f_2\rangle,\\
\end{align*}
where $f_2(z)=f(z/2)$ for each $z$. But by Cauchy-Schwarz and the fact that $f$ takes values of modulus at most 1,
\begin{align*}
|\langle f*f,f_2\rangle|^2&\leq\|f*f\|_2^2\|f_2\|_2^2\\
&\leq\E_{x+y=z+w}f(x)f(y)\overline{f(z)f(w)}\\
&=\|f\|_{U^2}^4.\\
\end{align*}
Therefore, if $\|f\|_{U^2}$ is small, then $\E_{x+y=2z}A(x)A(y)A(z)\approx\a^3$. 

In due course, we shall see how the arguments given above are more amenable to generalization than the Fourier-analytic proof we gave earlier.

We close this section by remarking that the definition of the $U^2$ norm and the basic observations we have made about it work just as well in an arbitrary finite Abelian group, and several of its properties hold even for non-Abelian groups.

\section{Generalization to matrices}

Given a function $f$ we can define a linear map $T_f$ that takes a function $g$ to the convolution $f*g$. That is, we have
\[T_f(g)(x)=\E_{u}f(x-u)g(u).\]
If we define a matrix $M_f$ by $M_f(x,u)=f(x-u)$, then this formula becomes
\[T_f(g)(x)=\E_uM_f(x,u)g(u),\]
which is just the usual formula for multiplying a matrix by a vector, except that instead of summing over $u$ we have taken the expectation. It will be convenient, for the purposes of this section, to adopt a non-standard definition of matrix multiplication by using this normalization. That is, we will say that if $A$ and $B$ are two matrices, then 
\[(AB)(x,z)=\E_yA(x,y)B(y,z).\]
Since 
\[T_fT_gh=T_f(g*h)=f*(g*h)=(f*g)*h,\]
we get that $M_fM_g=M_{f*g}$ with this normalization.

Notice that 
\[T_f(\omega_r)(x)=f*\omega_r(x)=\E_uf(u)\omega^{r(x-u)}=\hat f(r)\omega_r(x).\]
Thus, $\omega_r$ is an eigenvector of $T_f$ with eigenvalue $\hat f(r)$. 

A more conceptual way of seeing this is to note that by the convolution identity, the convolution of $f$ with the function $g=\sum_r\hat g(r)\omega_r$ is the function $\sum_r\hat f(r)\hat g(r)\omega_r$, so with respect to the basis $\omega_0,\dots,\omega_{n-1}$ all convolution maps $T_f$ are multipliers (that is, given by diagonal matrices). 

These observations allow us to translate some of the concepts we have defined so far into matrix language. The Fourier coefficients of a function $f$ become the eigenvalues of the matrix $M_f$. However, that is just the beginning. Let us write $a\otimes b$ for the rank-1 matrix with 
\[(a\otimes b)(u,v)=a(u)b(v).\]
Note that if $a,b,f:\Z_n\to\C$, then
\[(a\otimes \overline{b})(f)(x)=a(x)\E_yf(y)\overline{b(y)}=a(x)\langle f,b\rangle.\]
Thus, the diagonalization of $f$ is telling us that 
\[M_f=\sum_r\hat f(r)\omega_r\otimes\overline{\omega_r},\]
since if we apply either side to the function $\omega_s$ we obtain $\hat f(s)\omega_s$.

We are now in a position to write down Parseval's identity in matrix terms. First, note that
\[\E_{x,y}|M_f(x,y)|^2=\E_{x,y}|f(x-y)|^2=\E_x|f(x)|^2=\|f\|_2^2.\]
Therefore, by Parseval's identity, we find that 
\[\E_{x,y}|M_f(x,y)|^2=\sum_r|\hat f(r)|^2.\]
The left-hand side is the $L_2$ norm of the matrix entries of $M_f$, which is often known as the (normalized) \emph{Hilbert-Schmidt norm}. And the right-hand side, though it appears to be expressed in terms of $f$, can be thought of as the sum of squares of the eigenvalues of $M_f$. 

This connection can be generalized to all matrices that have an orthonormal basis $u_1,\dots,u_n$ of eigenvectors. In that case we can write $M=\sum_i\lambda_iu_i\otimes\overline{u_i}$ and we find that
\begin{align*}\E_{x,y}|M(x,y)|^2&=\E_{x,y}\sum_{i,j}\lambda_i\overline{\lambda_j}u_i(x)\overline{u_i(y)u_j(x)}u_j(y)\\
&=\sum_{i,j}\lambda_i\overline{\lambda_j}\E_{x,y}u_i(x)\overline{u_i(y)u_j(x)}u_j(y)\\
&=\sum_{i,j}\lambda_i\overline{\lambda_j}|\langle u_i,u_j\rangle|^2\\
&=\sum_i|\lambda_i|^2.\\
\end{align*}
More generally still, if $M$ does not have an orthonormal basis of eigenvectors, it will still have a \emph{singular value decomposition}, that is, a decomposition of the form $\sum_i\lambda_iu_i\otimes\overline{v_i}$ where $(u_i)_1^n$ and $(v_i)_1^n$ are both orthonormal bases and the $\lambda_i$ are non-negative real numbers. (The non-negativity can be obtained by multiplying the $v_i$ by suitable scalars of modulus 1.) The above argument carries over with very little change, and we find that $\|M\|_2^2$ (that is, the square of the normalized Hilbert-Schmidt norm) is equal to the sum of the squares of the singular values. 

As we have already made clear, this fact specializes to Parseval's identity when the matrix is the matrix $M_f$ of a convolution operator $T_f$. 

More importantly, singular values of matrices play a rather similar role in graph theory to the role played by Fourier coefficients in additive combinatorics. To see this, let us first find an analogue for matrices of the $U^2$ norm. Given the correspondence so far, it should be equal to the $\ell_4$ norm of the singular values, and its fourth power should have a nice interpretation in terms of the matrix values. This does indeed turn out to be the case. An argument similar to the one just given for the Hilbert-Schmidt norm, but slightly more complicated, shows that
\[\sum_i|\lambda_i|^4=\E_{x,y,a,b}M(x,y)\overline{M(x+a,y)M(x,y+b)}M(x+a,y+b).\]
Now the fourth root of the left-hand side is a well-known matrix norm -- the fourth-power trace class norm. From this one can deduce that the fourth root of the right-hand side is a norm, which we write as $\|M\|_\square$ and call the \emph{box norm} (because we are summing over aligned rectangles). But as with the $U^2$ norm, one can prove this fact directly by first defining a generalized inner product for two-variable functions
\[[f_1,f_2,f_3,f_4]=\E_{x,y,a,b}f_1(x,y)\overline{f_2(x+a,y)f_3(x,y+b)}f_4(x+a,y+b),\]
using the Cauchy-Schwarz inequality to prove that
\[[f_1,f_2,f_3,f_4]\leq\|f_1\|_\square\|f_2\|_\square\|f_3\|_\square\|f_4\|_\square,\]
and finally deducing that $\|f+g\|_\square^4\leq(\|f\|_\square+\|g\|_\square)^4$ in more or less the same way as we did for the $U^2$ norm.

After this it will come as no surprise to learn that the box norm specializes to the $U^2$ norm when the matrix is a Toeplitz matrix (that is, the matrix of a convolution operator). Indeed, we have that
\begin{align*}\|M_f\|_\square^4&=\E_{x,y,a,b}f(x-y)\overline{f(x+a-y)f(x-y-b)}f(x+a-y-b)\\
&=\E_{x,a,b}f(x)\overline{f(x+a)f(x-b)}f(x+a-b)\\
&=\E_{x,a,b}f(x)\overline{f(x+a)f(x+b)}f(x+a+b)\\
&=\|f\|_{U^2}^4.\\
\end{align*}
Of course, we could also have deduced this less directly by using the relationship between eigenvalues, Fourier coefficients, and the two norms.

Now let us take a graph $G$ and let $M$ be its adjacency matrix. (That is, $M(x,y)=1$ if there is an edge from $x$ to $y$ and 0 otherwise.) Then the analogy between subsets of $\Z_n$ (or more general finite Abelian groups) and matrices strongly suggests that the box norm $\|.\|_\square$ should be a useful measure of quasirandomness. That is indeed the case. If $G$ has density $\d$, meaning that $\E_{x,y}M(x,y)=\d$, then a straightforward argument using the Cauchy-Schwarz inequality shows that $\|M\|_\square\geq\d$. If equality almost holds, then $G$ turns out to enjoy a number of properties that typical random graphs have. 

To see this, we begin by noting that the box norm relates to the largest singular value in much the way that the $U^2$ norm relates to the largest Fourier coefficient. If the singular values are $\lambda_1,\dots,\lambda_n$ and if $\lambda=(\lambda_1,\dots,\lambda_n)$, then
\[\|\lambda\|_\infty^4\leq\|\lambda\|_4^4\leq\|\lambda\|_2^2\|\lambda\|_\infty^2,\]
and if the matrix entries have modulus at most 1 then we know in addition that $\|\lambda\|_2^2=\|M\|_2^2\leq 1$. Therefore, the largest singular value (which is equal to the operator norm of the matrix) is small if and only if the box norm is small.

For convenience let us now assume that $G$ is regular, so every vertex has degree $\d n$. (This is not a major assumption, but the statements become slightly less clean and the proofs slightly more complicated if we do not make it.) Then the constant function $u(x)=1$ is an eigenvector of $M$ with eigenvalue $\d$. (Recall that we are using expectations in our matrix multiplication, which is why we get $\d$ here rather than $\d n$.)

Now consider the matrix $A=M-\d u\otimes u$. That is, $A(x,y)=M(x,y)-\d$. Since $G$ is regular, we find that $\E_xA(x,y)=0$ for every $y$ and $\E_yA(x,y)=0$ for every $x$. From this it is not hard to prove that $\|M\|_\square^4=\d^4+\|A\|_\square^4$: we expand $\|A+\d u\otimes u\|_\square^4$ as a sum of sixteen terms and the only ones that are not zero are the term with all $A$s and the term with all $\d$s.

Therefore, if $\|M\|_\square$ is close to $\d$, it follows that $\|A\|_\square$ is close to zero, which implies that the largest singular value of $A$ is small, and therefore that $A$ has a small operator norm. Let $\theta$ be this operator norm.

Now let $f$ and $g$ be two functions defined on the vertex set of $G$ that take values in the interval $[-1,1]$. Then 
\[|\langle Af,g\rangle|\leq\|Af\|_2\|g\|_2\leq\theta\|f\|_2\|g\|_2\leq\theta.\] 
We also have that
\[\langle(\d u\otimes u)(f),g\rangle=\langle(\d\E_xf(x))u,g\rangle=\d\E_xf(x)\E_yg(y).\]
It follows that
\[|\langle Mf,g\rangle-\d\E_xf(x)\E_yg(y)|\leq\theta.\]
But $\langle Mf,g\rangle=\E_{x,y}M(x,y)f(x)g(y)$, so if $\theta$ is small then this is telling us that
\[\E_{x,y}M(x,y)f(x)g(y)\approx\d\E_{x,y}f(x)g(y).\]
Suppose now that $f$ and $g$ are the characteristic functions of sets $U$ and $V$ of density $\a$ and $\b$. Now we have that
\[\E_{x,y}M(x,y)U(x)V(y)\approx\d\E_{x,y}U(x)V(y)=\d\a\b.\]
This tells us that the number of edges from $U$ to $V$ in the graph is approximately $\d|U||V|$, which is exactly the number one would expect if $G$ was a random graph with density $\d$. 

Now the fourth power of the box norm of $M$ can be seen to equal the \emph{4-cycle density} of the graph $G$, that is, the probability, if vertices $x_1,x_2,x_3,x_4$ are chosen independently at random, that $x_1x_2,x_2x_3,x_3x_4$ and $x_4x_1$ are all edges of $G$. Thus, we have started with a ``local'' assumption -- that the number of 4-cycles in the graph is almost as small as it can possibly be given the density of the graph -- and ended up with a global conclusion -- that the number of edges between any two large sets is approximately what one would expect in a random graph of the same density. This fact has many applications in graph theory. 

The converse can also be shown without too much difficulty. In fact, there turn out to be several properties that are all loosely equivalent and all say that in one way or another a graph $G$ behaves like a random graph. A particularly interesting one from the point of view of comparison with Roth's theorem is the statement that if a graph $G$ of density $\d$ is quasirandom (in, for example, the sense of having box norm approximately $\d$) then for any graph $H$ with $k$ edges (here $k$ is fixed and the size of $G$ is tending to infinity) the $H$ density in $G$ is approximately $\d^k$, as it would be in a random graph. Conversely, if $G$ contains the ``wrong" number of copies of $H$, then we can find a subgraph that is substantially denser than the original graph.

The theory of quasirandom graphs goes back to papers of Thomason \cite{thomason87} and Chung, Graham and Wilson \cite{chunggrahamwilson89}. It has subsequently been generalized in many directions and to many other mathematical structures, and quasirandomness has become a major theme in mathematics. (Of course, in other guises, this theme has existed for much longer: one has only to think of the distribution of prime numbers, for instance.) 

It is important to point out that not all the basic properties of the Fourier transform carry over in a nice way to matrices. For example, the inner product corresponding to the normalized Hilbert-Schmidt norm is 
\[\langle A,B\rangle=\E_{xy}A(x,y)\overline{B(x,y)}=\tr(AB^*)\]
(where the trace here is also defined in a normalized way -- that is, $\tr(A)=\E_xA_{xx}$). If the singular-value decompositions of $A$ and $B$ are $\sum_i\lambda_iu_i\otimes\overline{v_i}$ and $\sum_j\mu_jw_j\otimes\overline{z_j}$, then $\langle A,B\rangle$ works out to be
\[\sum_{i,j}\lambda_i\overline{\mu_j}\langle u_i,w_j\rangle\overline{\langle v_i,z_j\rangle}.\]
If it happens that $u_i=w_i$ and $v_i=z_i$ for every $i$, as it does when $A=B$, then this simplifies to $\sum_i\lambda_i\overline{\mu_i}$, the formula we would likeif we wanted a direct analogue of Parseval's identity, but if not then we have to make do with the more complicated formula above (which nevertheless can be useful sometimes). 

Similarly, there is no tidy analogue of the convolution identity except under very special circumstances. In general,
\[(\sum_i\lambda_iu_i\otimes\overline{v_i})(\sum_j\mu_jw_j\otimes{z_j})=\sum_{i,j}\lambda_i\mu_j\overline{\langle w_j,v_i\rangle}u_i\otimes\overline{z_j}.\]
If $v_i=w_i$ for every $i$, then this simplifies to $\sum_i\lambda_i\mu_iu_i\otimes\overline{z_i}$, so we find that the singular values of the matrix product are products of the singular values of the original matrices. But this is an unusual situation (that happens to occur when the two matrices are convolution matrices and all the bases are the same basis of trigonometric functions). 

\section{Quadratic Fourier analysis} \label{quadfanal}

In this section I shall discuss a generalization of Fourier analysis that lacks a satisfactory inversion formula. The inversion formula might seem to be such a fundamental property of the Fourier transform that the generalization does not deserve to be called a generalization of Fourier analysis. However, for several applications of Fourier analysis, a weaker property suffices, and that weaker property can be generalized. Nevertheless, it is a very interesting open problem to develop the theory further so as to make the analogy with conventional discrete Fourier analysis closer. 

Let us begin by looking at a problem that demonstrates the need for a generalization at all, namely Szemer\'edi's theorem for progressions of length~4. It is natural to try to model a proof on the proof for progressions of length 3. At the heart of that proof is the identity
\[\E_{x+y=2z}f(x)g(y)h(z)=\sum_r\hat f(r)\hat g(r)\hat h(-2r).\]
We have essentially proved this already, but a variant of the argument is to observe that both sides are equal to $\E_{x,y,z}f(x)g(y)h(z)\sum_r\omega^{-r(x+y-2z)}$. So it is natural to look for a similar identity for progressions of length 4. Such a progression can be thought of as a quadruple $(x,y,z,w)$ such that $x+z=2y$ and $y+w=2z$. However,
\begin{align*}\E_{x+z=2y,y+w=2z}&f_1(x)f_2(y)f_3(z)f_4(w)\\
&=\E_{x,y,z,w}f_1(x)f_2(y)f_3(z)f_4(w)\sum_{r,s}\omega^{-r(x-2y+z)-s(y-2z+w)}\\
&=\sum_{r,s}\hat{f_1}(r)\hat{f_2}(-2r+s)\hat{f_3}(r-2s)\hat{f_4}(s).\\
\end{align*}
A quadruple $(a,b,c,d)$ can be written in the form $(r,-2r+s,r-2s,s)$ if and only if $3a+2b+c=b+2c+3d=0$. So we have ended up with a sum over four variables that satisfy two linear equations, which is what we had before we took the Fourier transform. So we have not gained anything.

An even more compelling argument that the Fourier transform is too blunt a tool for our purposes is to note that it is possible for all the Fourier coefficients of $f_1,f_2,f_3$ and $f_4$ to be tiny, but for the expectation $\E_{x,d}f_1(x)f_2(x+d)f_3(x+2d)f_4(x+3d)$ to be large. (This is another way of writing the left-hand side of the equality above.) Let $f_1(x)=\omega^{x^2}$, $f_2(x)=\omega^{-3x^2}$, $f_3(x)=\omega^{3x^2}$ and $f_4(x)=\omega^{-x^2}$. Then
\[\E_{x,d}f_1(x)f_2(x+d)f_3(x+2d)f_4(x+3d)=\E_{x,d}\omega^{x^2-3(x+d)^2+3(x+2d)^2-(x+3d)^2}.\]
But the exponent on the right-hand side is identically zero, so both sides are equal to 1, which is as large as the expectation can possibly be given that all four functions take values of modulus 1. On the other hand, functions like $\omega^{x^2}$ have tiny Fourier coefficients. To see this (assuming for convenience that $n$ is odd), note that if $f(x)=\omega^{x^2}$, then
\[\hat f(r)=\E_x\omega^{x^2-rx}=\E_x\omega^{(x-r/2)^2-r^2/4}=\omega^{-r^2/4}\E_x\omega^{x^2}.\]
This shows that $|\hat f(r)|=|\E_x\omega^{x^2}|$ is the same for all $r$, and therefore by Parseval it equals $n^{-1/2}$ for all $r$. In other words, the largest Fourier coefficient is as small as Parseval's identity will allow.

It is almost impossible at this stage not to have the following thought. For Roth's theorem, the functions that caused trouble by not being sufficiently random-like were the trigonometric functions $x\mapsto\omega^{rx}$. These are \emph{linear phase functions} -- that is, compositions of linear functions with the function $x\mapsto\omega^x$. We have just seen that when it comes to discussing arithmetic progressions of length 4, quadratic phase functions, that is, functions of the form $\omega^{q(x)}$ where $q$ is a quadratic, cause problems. Could it be that these are somehow the only functions that cause problems? Does there exist some kind of ``quadratic Fourier analysis" that allows one to expand a function as a linear combination of quadratic phase functions and thereby to generalize the proof of Roth's theorem to progressions of length~4?

The answer to this question turns out to be a partial yes. More precisely, one can generalize ``linear" Fourier analysis by just enough to obtain a proof of Szemer\'edi's theorem for progressions of length 4, but the generalized Fourier analysis lacks some of the nice properties of the usual Fourier transform, as a result of which the proof becomes substantially harder. In particular, it turns out that the quadratic phase functions are \emph{not} the only ones that cause trouble -- there are also some more general functions that exhibit sufficiently quadratic-like behaviour to cause problems similar to the ones caused by the ``pure" quadratic phase functions. But before we get on to that, it will be useful to look at another concept that comes into the picture.

\section {The $U^3$ norm} \label{U3}

Discrete Fourier analysis decomposes a function into characters. It is far from obvious how to define a ``quadratic" analogue of this decomposition, since one's natural first guesses turn out not to have the properties one wants, as we shall see later. But right from the start it is clear that there are problems, because there are $n^2$ functions of the form $x\mapsto\omega^{ax^2+bx}$, so we cannot hope to define a quadratic Fourier transform by simply writing down a suitable basis of $\C^n$ and expanding functions in terms of that basis.

It is for this reason that the reformulation of the norm $f\mapsto\|\hat f\|_4$ in purely physical-space terms is so important. It gives us a concept that \emph{is} easy to generalize. As one might expect, there are $U^k$ norms for all $k\geq 2$ (and also a seminorm when $k=1$), but since it is clear what they are once one has seen the $U^3$ norm, we shall present just that. It is defined by the formula
\begin{align*}
\|f\|_{U^3}^8=\E_{x,a,b,c}f(x)\overline{f(x+a)f(x+b)}&f(x+a+b)\overline{f(x+c)}\\
&f(x+a+c)f(x+b+c)\overline{f(x+a+b+c)}\\
\end{align*}
That is, where the $U^2$ norm involves an average over ``squares", the $U^3$ norm involves a similar average over ``cubes". (The $U^k$ norm involves a similar average over $k$-dimensional cubes.) The letter U stands for ``uniformity", because when a function has a small uniformity norm, its values are ``uniformly distributed" in a useful sense.

There are a few remarks to make about the $U^3$ norm to give an idea of its basic properties and to indicate why it is likely to be important to us.
\begin{itemize}
\item First, it really is a norm. This is proved in much the same way as it is for the $U^2$ norm: one defines an appropriate generalized inner product (by using eight different functions in the formula above instead of just one), deduces a generalized Cauchy-Schwarz inequality from the conventional Cauchy-Schwarz inequality, and finally deduces a generalized Minkowski inequality from the generalized Cauchy-Schwarz inequality.
\item Secondly, if $f$ is a quadratic phase function $f(x)=\omega^{rx^2+sx}$, then $\|f\|_{U^3}$ takes the largest possible value (given that all the values of $f$ have modulus 1), namely 1. This is simple to check, and boils down to the fact that
\begin{align*}x^2-(x+a)^2-(x+b)^2&+(x+a+b)^2-(x+c)^2\\
&+(x+a+c)^2+(x+b+c)^2-(x+a+b+c)^2=0\\
\end{align*}
for every $x,a,b$ and $c$.
\item Thirdly, the $U^k$ norms increase as $k$ increases. In particular, the $U^3$ norm is larger than the $U^2$ norm. This means that the statement that $\|f\|_{U^3}$ is small is stronger than the statement that $\|f\|_{U^2}$ is small. That fact, combined with the observation that $\|f\|_{U^3}$ is large for quadratic phase functions, gives some reason to hope that the $U^3$ norm could be a useful measure of quasirandomness for Szemer\'edi's theorem for progressions of length 4.
\item Fourthly, if $A$ is a set of density $\a$, then an easy Cauchy-Schwarz argument shows that $\|A\|_{U^3}\geq\a$. Also, $\|A\|_{U^3}^8$ counts the number of ``cubes" in $A$. So when we talk about sets, we will want to regard a set as ``quadratically uniform" if it has almost the minimum number of cubes. This will be a stronger property than the ``linear uniformity" that we used in the proof of Roth's theorem, which is based on the number of squares.
\end{itemize}

Presenting those remarks is slightly misleading, however, as it suggests that the definition of the $U^3$ norm is a purely speculative generalization of the definition of the $U^2$ norm that just happens to be useful. In fact, the definition arises naturally (or at least can arise naturally) when one tries to generalize the physical-space argument we saw earlier that shows that a set with small $U^2$ norm has roughly the expected number of arithmetic progressions of length 3. One ends up being able to show that if $f_1,f_2,f_3$ and $f_4$ are functions that take values of modulus at most 1, then 
\[|\E_{x,d}f_1(x)f_2(x+d)f_3(x+2d)f_4(x+3d)|\leq\min_i\|f_i\|_{U^3}.\]
In other words, if one of the four functions has a small $U^3$ norm, then the arithmetic progression count must be small.

The point I am making here is that if one sets out to prove a bound for the left-hand side in terms of some suitable function of $f_4$, say, knowing that one's main tool is the Cauchy-Schwarz inequality, then the function that one obtains is precisely the $U^3$ norm.

The inequality above can be used to show that if $A$ is a set of density $\a$ and $\|A\|_{U^3}\leq\a+c(\a)$, then $A$ is sufficiently quasirandom to contain an arithmetic progression of length 4, and in fact to have 4-AP density approximately $\a^4$. To prove this, one writes $A=\a+f$ with $\|f\|_{U^3}$ small, one expands out the expression
\[\E_{x,d}A(x)A(x+d)A(x+2d)A(x+3d)\]
as a sum of 16 terms, and one uses the inequality above to show that all these terms are small apart from the main term $\a^4$. 

\section{Generalized quadratic phase functions}

In the previous section we noted that if $q$ is a quadratic function defined on $\Z_n$, and $f$ is the function $f(x)=\omega^{q(x)}$, then $\|f\|_{U^3}=1$, which is as large as it can possibly be. The key to this fact, as we have already noted, is that quadratic functions have the property that
\begin{align*} q(x)-q(x+a)-q(x+b)&+q(x+a+b)-q(x+c)\\
&+q(x+a+c)+q(x+b+c)-q(x+a+b+c)=0\\
\end{align*}
for every $x,a,b,c$. Moreover, this property characterizes quadratic functions.

However, if we do not insist on maximizing $\|f\|_{U^3}$ but merely getting \emph{close} to the maximum, then we suddenly let in a whole lot more functions. In this section I shall describe one or two of them.

There is a general recipe for producing them, which is to take a set $A\subset\Z_n$ and construct a \emph{quadratic homomorphism} on $A$ -- that is, a map $\psi:A\to\C$ that takes values of modulus 1 and satisfies the equation 
\[\psi(x)\overline{\psi(x+a)\psi(x+b)}\psi(x+a+b)\overline{\psi(x+c)}\psi(x+a+c)\psi(x+b+c)\overline{\psi(x+a+b+c)}=1\]
whenever all of $x, x+a, x+b, x+c, x+a+b, x+a+c, x+b+c$ and $x+a+b+c$ belong to $A$. (As we have already noted, if $A$ has density $\a$, there will be at least $\a^8n^4$ ``cubes" of this kind.) We then define $f(x)$ to be $\psi(x)$ for $x\in A$ and 0 otherwise. For this to produce interesting examples, we need to choose our set $A$ carefully, but that can be done.

As a first example, take $A$ to be the set $\{1,\dots,\lfloor n/2\rfloor\}$. If we now let $\b$ be any real number, we can define $f(x)$ to be $e^{2\pi i\b x^2}$ on $A$ and zero outside. If $\b$ is a multiple of $1/n$, then this will give us a function $\omega^{rx^2}$ restricted to $A$. However, if we choose $\b$ not to be close to a multiple of $1/n$ we can obtain functions that do not even correlate with functions of the form $\omega^{rx^2+sx}$. Suppose, for example, that we take $\b=1/2n$. Then our best chance of a correlation will be with either the constant function 1 or the function $\omega^{x^2}=e^{4\pi i \b x^2}$. In both cases, the inner product has modulus $n^{-1}|\sum_{x\in A}e^{\pi i x^2/n}|$, which can be shown to be small by a simple trick known as Weyl differencing: we observe that
\[|\sum_{x\in A}e^{\pi i x^2/n}|^2=\sum_{x,y}e^{\pi i (x^2-y^2)/n}=\sum_{x,y}e^{\pi i (x+y)(x-y)}.\]
The last sum can be split into a sum of geometric progressions, each of which can be evaluated explicitly, and almost all of which turn out to be small. Essentially the same technique proves that in fact our function $f$ has a very small correlation with any function of the form $\omega^{q(x)}$ for a quadratic function $q$ defined on $\Z_n$. 

It is worth stopping to think about why a similar argument does not show that we have to consider more functions even in the linear case. What if we take a function on the set $A$ above of the form $e^{2\pi i \b x}$ with $\b$ far from a multiple of $1/n$? In fact, what if we take $\b=1/2n$ as before?

In this case the correlation with a constant function has magnitude $n^{-1}|\sum_{x\in A}e^{\pi i x/n}|$, and $x/n$ lies between $0$ and $1/2$. It follows that all the numbers $e^{\pi x/n}$ are on one side of the unit circle, and the result is that we do not get the cancellation that occurred with the quadratic example above. The difference between the two situations is that the function $e^{\pi i x^2/n}$ jumps round the circle many times, whereas the function $e^{\pi i x/n}$ does not -- which is due to the fact that the function $x^2$ grows much more rapidly than the function $x$. 

Another way of choosing a set $A$ is to make it look like a portion of $\Z^d$ for some small $d$. To give an example with $d=2$, let $m=\lfloor \sqrt n/2\rfloor$ and let $A$ consist of all numbers of the form $x+2my$ such that $x,y\in\{0,1,\dots,m-1\}$. This we can think of as a two-dimensional set with basis $1$ and $2m$: the pair $(x,y)$ then represents the point $x+2my$ in coordinate form. 

An obvious class of functions to take on a multidimensional set is the class of quadratic forms, and we can do that here. We pick coefficients $a,b,c\in\Z_n$ and define $f(x+2my)$ to be $\omega^{ax^2+bxy+cy^2}$ for all $x,y\in\{0,1,\dots,m-1\}$ and take all other values of $f$ to be zero. It is easy to check that $f$ is a quadratic homomorphism in the sense just defined, and it can also be shown that $f$ does not correlate with any pure quadratic phase function.

We can of course combine these ideas by taking more general coefficients. We can also define a wide variety of two-dimensional sets by taking different ``basis vectors", and we can increase the dimension. Thus, the set of functions we are forced to consider is much richer than the corresponding set for the $U^2$ norm.

\section{Szemer\'edi's theorem for progressions of length 4}

We remarked at the end of Section \ref{U3} that if a set $A$ is quasirandom in the sense of having an almost minimal $U^3$ norm, then it contains an arithmetic progression of length 4. Furthermore, the proof of this fact is closely analogous to the proof of the corresponding fact relating the $U^2$ norm to arithmetic progressions of length 3. So it is natural to try to continue the analogy and complete a proof of Szemer\'edi's theorem for progressions of length 4. That is, we would like to argue that if the $U^3$ norm of $A$ is \emph{not} approximately minimal, then we can obtain a density increase on an appropriate subspace. 

At this point we find that we are a little stuck. In the $U^2$ case we used the fact that if $f$ is a function taking values of modulus at most 1, and $\|f\|_{U^2}=\|\hat f\|_4$ is bounded below by a positive constant $c$, then $\|\hat f\|_\infty$ is bounded below by $c^2$, which we can use to argue that a set with no arithmetic progression of length 3 must be sufficiently ``unrandom" to correlate well with a trigonometric function. So to continue the analogy, it looks as though we need to find norms $\|.\|$ and $|||.|||$ (here $\|f\|$ and $|||f|||$ are the hoped-for analogues of $\|\hat f\|_4$ and $\|\hat f\|_\infty$, respectively) with the following properties.
\begin{enumerate}
\item The norm $\|.\|$ is defined in a different way from the $U^3$ norm, but happens to be equal to it.
\item If $\|f\|_\infty\leq 1$ and $\|f\|\geq c$, then one can prove very straightforwardly that $|||f|||\geq\g(c)$ (where $\g(c)>0$ if $c>0$, and ideally the dependence will be a good one). 
\item The fact that $|||f|||\geq\g$ is telling us that there is some function $\psi\in\Psi$ for which $|\langle f,\psi\rangle|\geq\theta(\gamma)$, where $\Psi$ is a class of ``nice" functions (which will probably exhibit behaviour similar to that of quadratic phase funtions). 
\item If $A$ is a set of density $\a$, $f=A-\a$, and $|\langle f,\psi\rangle|\geq\theta$ for some $\psi\in\Psi$, then there is a long subprogression $P$ inside which $A$ has density at least $\alpha+\eta(\theta)$. 
\end{enumerate}
Implicit in the third of these conditions is that $|||.|||$ and $\Psi$ are related by the formula $\max\{|\langle f,\psi\rangle|:\psi\in\Psi\}$. 

The big problem we face is that there is no obvious reformulation of the $U^3$ norm analogous to the reformulation $\|f\|_{U^2}=\|\hat f\|_4$ of the $U^2$ norm. So we do not know of a candidate for $\|.\|$. However, that does not mean that there is nothing we can do, since there is still the possibility of passing directly from the statement that $\|f\|_{U^3}\geq c$ to the statement that $|\langle f,\psi\rangle|\geq\theta(c)$ for some suitably nice function $\psi$, or even bypassing this statement and heading straight for the conclusion that $A$ is denser in some long subprogression. Both approaches turn out to be possible. 

It is not possible here to do more than give a very brief sketch of how the proof works. We start with a function $f$ with $\|f\|_\infty\leq 1$ and $\|f\|_{U^3}^8\geq\g$. That inequality expands to the inequality
\begin{align*}
\E_{x,a,b,c}f(x)\overline{f(x-a)f(x-b)}&f(x-a-b)\overline{f(x-c)}\\
&f(x-a-c)f(x-b-c)\overline{f(x-a-b-c)}\geq\g,\\
\end{align*}
where we have switched from plus signs to minus signs for unimportant aesthetic reasons. We now define, for each $a$, a function $\partial_af$ by the formula $\partial_af(x)=f(x)\overline{f(x-a)}$, which allows us to rewrite the inequality above as
\[\E_a\E_{x,b,c}\partial_af(x)\overline{\partial_af(x-b)\partial_af(x-c)}\partial_af(x-b-c)\geq\gamma.\]
Now this is just telling us that $\E_a\|\partial_af\|_{U^2}^4\geq\theta$, from which it follows that there must be several $a$ for which $\|\partial_af\|_{U^2}$ is large. By the rough equivalence of the $U^2$ norm with the magnitude of the largest Fourier coefficient, we can deduce from this that several of the functions $\partial_af$ have at least one large Fourier coefficient. It follows that there is a large set $B$ and a function $\phi:B\to\Z_n$ such that $\widehat{\partial_af}(\phi(a))$ is large for every $a\in B$. More formally, we can obtain an inequality 
\[\E_a B(a)|\widehat{\partial_af}(\phi(a))|^2\geq\theta\]
for some $\theta$ that depends (polynomially) on $\gamma$. 

It turns out that one can perform some algebraic manipulations with this statement and eventually prove that the function $\phi$ has an interesting ``partial additivity" property, which states that there are at least $\eta n^3$ quadruples $(x,y,z,w)\in B^4$ (for some $\eta$ that depends on $\gamma$ only) such that
\[x+y=z+w\]
and
\[\phi(x)+\phi(y)=\phi(z)+\phi(w).\]
This property appears at first to be somewhat weak, since it tells us that $\phi$ is additive on only a small percentage of the quadruples $x+y=z+w$. Remarkably, however, this is another instance where a local assumption can be used to prove a global conclusion: the only way that $\phi$ can be this additive is if it has a form that can be described very precisely.

Recall the two-dimensional set we defined in the previous section. It is an example of a \emph{two-dimensional arithmetic progression}. More generally, a $k$-dimensional arithmetic progression is a set of the form
\[\{x+a_1d_1+a_2d_2+\dots+a_kd_k:0\leq a_i<m_i\}.\]
The numbers $d_1,\dots,d_k$ are the common differences and the numbers $m_1,\dots,m_k$ are the lengths. The arithmetic progression is called \emph{proper} if it has cardinality $m_1\dots m_k$ -- that is, no two of the $a_1d_1+\dots+a_kd_k$ coincide. 

Given such a progression and coefficients $\mu_0,\mu_1,\dots,\mu_k\in\Z_n$ one can define something like a linear form by the obvious formula
\[x+a_1d_1+a_2d_2+\dots+a_kd_k\mapsto\mu_0+\sum_i\mu_ia_i.\]
Let us call such a map \emph{quasilinear}. 

The result that tells us about the structure of $\phi$ is the following.

\begin{theorem}
For every $\eta>0$ there is an integer $d=d(\eta)$ and a constant $\zeta=\zeta(\eta)>0$ with the following properties. Let $B\subset\Z_n$ and suppose that there are $\eta n^3$ quadruples $(x,y,z,w)\in B^4$ with $x+y=z+w$ and $\phi(x)+\phi(y)=\phi(z)+\phi(w)$. Then there is a proper arithmetic progression $P$ of dimension at most $d$ and a quasilinear map $\psi:P\to\Z_n$ such that for at least $\zeta n$ values of $x\in\Z_n$ we have that $x\in B\cap P$ and $\phi(x)=\psi(x)$. 
\end{theorem}

\noindent Loosely speaking, this tells us that there must be a quasilinear map that agrees a lot of the time with $\phi$. To prove this, one must use some important results in additive combinatorics, such as a famous theorem of Freiman \cite{freiman93} (and more particularly a proof of the theorem due to Ruzsa \cite{ruzsa94}) as well as a quantitative version \cite{gowers01} of a theorem of Balog and Szemer\'edi \cite{balogszem94}. 

Now let us see why it is plausible that linear behaviour of the function $\phi$ should lead to quadratic behaviour in the function $f$ from which it was derived. Consider an example where $f$ is defined by a formula of the form $f(x)=\omega^{\nu(x)}$. Then $\partial_af(x)=\omega^{\nu(x)-\nu(x-a)}$. So the statement that $\widehat{\partial_af}(\phi(a))$ is large is telling us that the functions $\omega^{\nu(x)-\nu(x-a)}$ and $\omega^{a\phi(x)}$ correlate well. Since $\phi$ exhibits linear behaviour, the function $(a,x)\mapsto a\phi(x)$ exhibits bilinear behaviour.

But that is exactly what happens when $\nu$ is a quadratic function: if $\nu(x)=rx^2+sx$, then $\nu(x)-\nu(x-a)=2rxa-ra^2+sa$, which implies that $\partial_af$ has a large Fourier coefficient at $2ra$. 

At this point one can use the information we have in a reasonably straightforward way to prove a weakish statement that is sufficient for Szemer\'edi's theorem, or we can work harder to prove a stronger statement that can be thought of as giving us some kind of quadratic Fourier analysis. The weakish statement (stated qualitatively) is the following.

\begin{lemma} \label{partition}
Let $f:\Z_n\to\C$ be a function with $\|f\|_\infty\leq 1$ and suppose that there exists a quasilinear function $\psi$ defined on a low-dimensional arithmetic progression $P$ such that $\widehat{\partial_af}(\psi(a))$ is large for many $a\in P$. Then there are long arithmetic progressions $P_1,\dots,P_m$ that partition $\Z_n$ and quadratic polynomials $q_1,\dots,q_m:\Z_n\to\Z_n$ such that 
\[n^{-1}\sum_i|\sum_{x\in P_i}f(x)\omega^{-q_i(x)}|\]
is bounded away from zero.
\end{lemma}

This tells us that on average $f$ correlates with quadratic phase functions on the arithmetic progressions $P_i$. From this result it turns out to be possible to deduce that there is a refined partition into smaller arithmetic progressions such that $f$ correlates on average with \emph{linear} phase functions, and then we are in essentially the situation we were in with Roth's theorem and can complete the proof of Szemer\'edi's theorem for progressions of length 4. 

This generalization to progressions of length 4 of the Fourier-analytic method of Roth was obtained by the author \cite{gowers98} and extended to progressions of all lengths in \cite{gowers01} (which includes a separate treatment of the length-4 case). 

\section{The inverse theorem for the $U^k$ norms}

From the point of view of generalizing Fourier analysis, however, Lemma \ref{partition} is unsatisfactory. Our previous deductions tell us that the hypothesis of the lemma holds when $f$ is a function with $\|f\|_\infty\leq 1$ and $\|f\|_{U^3}\geq c$, so the conclusion holds too. That gives us a lot of information about $f$, but it says nothing about how the quadratic polynomials $q_i$ might be related. It therefore gives us only local information about $f$, from which it is not possible to deduce a converse: just because $f$ correlates with quadratic phase functions on the progressions $P_i$, it does not follow that $\|f\|_{U^3}$ is large. (In fact, even constant functions do not do the job: if we were to choose for each $i$ a random $\e_i\in\{-1,1\}$ and set $f(x)$ to equal $\e_i$ everywhere on $P_i$, we would not have a function with large $U^3$ norm.) 

By contrast, if $\|f\|_{U^2}$ is large, then we obtain very simply that $\|\hat f\|_\infty$ is large, which tells us that $f$ correlates with a function of the form $\omega^{rx}$, and that, equally simply, implies that $\|f\|_{U^2}$ is large.

What we would really like is to get from the hypothesis of Lemma \ref{partition} to a more global conclusion, which would say that $f$ correlates with a generalized quadratic phase function of the kind described in the previous section. It is plausible that such a result should exist: from linear behaviour of the function $\phi$ one can deduce straightforwardly that $f$ correlates with a pure quadratic phase function, so if we have generalized linear behaviour (of a rather precise kind) then it seems reasonable to speculate that $f$ should correlate with a correspondingly generalized quadratic phase function. 

The main obstacle to proving this is that the function $(a,x)\mapsto a\phi(x)$ is not symmetric. If it were, then the proof would be fairly straightforward. However, Green and Tao found an ingenious ``symmetrization argument" that allowed them to deduce from the hypotheses of Lemma \ref{partition} a more symmetric set of hypotheses that yielded the desired result \cite{greentao08:2}. I shall state it somewhat imprecisely here. It is known as the inverse theorem for the $U^3$ norm.

\begin{theorem} \label{inverse}
For every $c>0$ there exists $c'>0$ with the following property. Let $f:\Z_n\to\C$ be a function with $\|f\|_\infty\leq 1$ and $\|f\|_{U^3}\geq c$. Then there exists a generalized quadratic phase function $g$ such that $\langle f,g\rangle\geq c'$. Conversely, every function that correlates well with a generalized quadratic phase function has a large $U^3$ norm.
\end{theorem}

The main imprecision is of course that I have not said exactly what a generalized quadratic phase function is. There are in fact several non-identical ways of defining them and the theorem is true for each one. The way I presented them in the previous section (where the exponent is something like a quadratic form on a multidimensional arithmetic progression) is perhaps the easiest to understand for a non-expert, but it is not the most convenient to use in proofs. 

A natural question to ask at this point is what happens for the $U^k$ norm when $k\geq 4$. If one is aiming for a generalization of Lemma \ref{partition}, and thereby for a proof of Szemer\'edi's theorem, the case $k=4$ (which corresponds to arithmetic progressions of length 5) is significantly harder than the case $k=3$, and after that the difficulty does not increase further. As for the inverse theorem, one would like to show that a function with large $U^k$ norm correlates well with a generalized polynomial phase function of degree $k-1$, but it is far from easy even to come up with a satisfactory definition of what such a function should be. 

To give an idea of the difficulty, here are a few examples. Recall that the proof we have just discussed involved, in an essential way, ``quasilinear" functions. A typical example of such a function is defined as follows. First choose $x_1,\dots,x_k\in\Z_n$ and positive integers $r_1,\dots,r_k$ such that $r_1r_2\dots r_k$ is comparable to $n$ and the $x_i$ are independent, in the sense that all the sums $\sum_ia_ix_i$ with $0\leq a_i<r_i$ are distinct. Now, given coefficients $c_1,\dots,c_k\in\Z_n$ we can define a partial function $\phi$ on $\Z_n$ by setting
\[\phi(a_1x_1+\dots+a_kx_k)=c_1a_1+\dots+c_ka_k.\]
This resembles a linear functional on a $k$-dimensional vector space. 

Given a set-up like this, there are various ways that we might try to define ``quasiquadratic" functions. One, which we have already discussed in a special case, is to use a formula such as
\[q(a_1x_1+\dots+a_kx_k)=\sum_{i,j}c_{ij}a_ia_j.\]
But there are other ways of doing it. For example, if $\phi$ is a quasilinear function, then we can look at a function such as $x\mapsto x\phi(x)$ or $x\mapsto\phi(x^2)$. Thus, there are many ways of mixing ``quasiness" with addition and multiplication to create ``quasipolynomials".

Given a partially defined function $\psi$ defined on $\Z_n$, we can then create a corresponding phase function $f$ by taking $f(x)$ to be $\omega^{\psi(x)}$ when $\psi(x)$ is defined, and 0 otherwise. When $\psi$ is a quasipolynomial of degree $k$, these phase functions often have large $U^{k+1}$ norm.

A more convenient language for discussing such functions is that of \emph{bracket polynomials}. These are real-valued functions defined on $\Z$ (but one can restrict them to intervals) built out of the arithmetic operations on $\R$ and the integer-part function. (Of course, once we have the integer-part function $x\mapsto\lfloor x\rfloor$ we also have the fractional-part function $x\mapsto\{x\}=x-\lfloor x\rfloor$.) A typical ``quadratic" example is a function of the form $x\mapsto ax\lfloor bx\rfloor$. More formally, a polynomial of degree $k$ over $\R$ is a bracket polynomial, if $\phi$ is a bracket polynomial of degree $k$, then so are $-\phi$ and $\{\phi\}$, and the sum and product of two bracket polynomials of at most a given degree is a bracket polynomial of at most the degree that you would get with ordinary polynomials. 

Unfortunately, bracket polynomials are not very easy to work with. For example, if $\phi$ is a bracket polynomial of degree $k$ and we define $f:\Z_n\to\C$ by treating each $x\in\Z_n$ as an element of the set $\{0,1,\dots,n-1\}$ and setting $f(x)=e(\phi(x))$ (where $e(t)$ is shorthand for $\exp(2\pi it)$), then it is reasonable to conjecture that $f$ always has a large $U^{k+1}$ norm. However, this seems not to be known: the best that we have is a result of Tointon \cite{tointon14}, who proves it when the starting polynomials have no constant term, and his result is not easy. It is also not a direct proof in the language of bracket polynomials. 

However, bracket polynomials are closely connected with a class of functions called \emph{nilsequences}, first introduced in \cite{bergelsonhostkra05}, that are easier to handle. These are defined as follows. (There are some choices about the details here -- for the sake of exposition I have opted for choices that make the definition as simple as possible, but slightly different choices are made in the work described below.)

Given any group $G$, it is Abelian if and only if it is equal to its commutator subgroup: the subgroup generated by the commutators $[x,y]=xyx^{-1}y^{-1}$. It is \emph{2-step nilpotent} if its commutators belong to the centre: that is, if $x$ commutes with $[y,z]$ for every $y$ and $z$. If that is not the case, then we obtain non-trivial elements of the form $x[y,z]x^{-1}[y,z]^{-1}$. If these all belong to the centre, then $G$ is \emph{3-step nilpotent}, and so on. Nilpotent groups, which can be thought of as groups that are close to being Abelian, play a central role in additive combinatorics.

A key example of an $s$-step nilpotent group is the Heisenberg group (over $\R$, say), which consists of all real $(s+1)\times(s+1)$ that are zero below the diagonal and 1 on the diagonal. It is easy to check that if $A$ is such a matrix, and $B$ is another one but with the property that $B_{ij}=0$ whenever $1\leq j-i\leq k$, then $[A,B]_{ij}=0$ whenever $1\leq j-i\leq k+1$. This proves that the group is indeed $s$-step nilpotent. In particular, when $s=2$ we have $3\times 3$ matrices that are zero below the diagonal and 1 on it: the commutator of two such matrices is equal to the identity except that there may be a non-zero entry in the top right-hand corner, and those matrices belong to the centre of the group.

Now let $G$ be a connected and simply connected Lie group, and suppose that it is $s$-step nilpotent. A good example to bear in mind is the Heisenberg example above -- in fact, even the Heisenberg example with $s=2$. Let $\Gamma$ be a discrete subgroup such that the quotient $G/\Gamma$ (consisting of left cosets of $\Gamma$) is compact. Such a quotient is called an \emph{$s$-step nilmanifold}. In the Heisenberg example, an obvious choice for $\Gamma$ is the set of matrices with integer entries. 

The group $G$ acts on $G/\Gamma$ by left multiplication, so given $x\in G/\Gamma$ and $g\in G$, we can form a sequence of iterates $x, gx, g^2x, \dots$. If $F$ is a continuous function from $G/\Gamma$ to $\R$, then the sequence $F(x), F(gx), F(g^2x),\dots$ is an \emph{$s$-step nilsequence}.

To see how this relates to bracket polynomials, let us look at the Heisenberg example and perform a couple of calculations. First, it is easy to prove by induction that
\[\begin{pmatrix}1&u&0\\ 0&1&v\\ 0&0&1\\ \end{pmatrix}^n=
\begin{pmatrix}1&nu&\frac 12n(n-1)uv\\ 0&1&nv\\ 0&0&1\\ \end{pmatrix}.\]
In general, the $n$th power of an element of the Heisenberg group will have a degree-$d$ dependence on $n$ for entries that are $d$ steps away from (and above) the main diagonal. So polynomials arise naturally.

Another calculation shows that brackets also arise naturally. Every element of $G$ can be written uniquely as a product $gh$ where $h\in\Gamma$ and $g$ belongs to a \emph{fundamental domain}. An obvious example of a fundamental domain in the Heisenberg case consists of all matrices for which the entry above the diagonal belongs to the interval $[0,1)$. Suppose now that we have a matrix $\begin{pmatrix}1&x&z\\ 0&1&y\\ 0&0&1\\ \end{pmatrix}$ and we want to decompose it in this way. It is not hard to pick the integer matrix that does the job: one chooses the entries just above the diagonal first, and then the top right-hand entry. The result of this exercise is to observe that
\[\begin{pmatrix}1&x&z\\ 0&1&y\\ 0&0&1\\ \end{pmatrix}
\begin{pmatrix}1&-\lfloor x\rfloor&-\lfloor z-x\lfloor y\rfloor\rfloor\\ 0&1&-\lfloor y\rfloor\\ 0&0&1\\ \end{pmatrix}
=\begin{pmatrix}1&\{x\}&\{z-x\lfloor y\rfloor\}\\ 0&1&\{y\}\\ 0&0&1\\ \end{pmatrix}.\]
Note that the second matrix has integer entries and the entries above the diagonal in the third matrix are in the interval $[0,1)$. 

Combining these two observations, we find that the representative of the coset of the matrix $\begin{pmatrix}1&u&0\\ 0&1&v\\ 0&0&1\\ \end{pmatrix}^n$ in the fundamental domain has top right-hand entry equal to 
\[\{\frac 12n(n-1)uv-nu\lfloor nv\rfloor\}.\]
Here $u$ and $v$ are fixed real numbers, so we have obtained a bracket polynomial in $n$.

One way of converting this bracket polynomial into a 2-step nilsequence would be to take a Lipschitz function defined on $[0,1)$ but supported on $[0,1/2]$ and to take the sequence $(a_n)$, where $a_n=F(\{\frac 12n(n-1)uv-nu\lfloor nv\rfloor\})$.

We are now ready for a statement of the inverse theorem. It was formulated by Green and Tao in \cite{greentao10:1}: their formulation was strongly influenced by important work of Host and Kra \cite{hostkra05}, as was the paper of Bergelson, Host and Kra \cite{bergelsonhostkra05}, which was where a link between nilsequences and $U^k$-norms (or rather an ergodic-theoretic analogue of $U^k$-norms) was first established. The proof of the inverse theorem, which is a milestone in the subject, is due to Green, Tao and Ziegler \cite{greentaoziegler12}. It completed a programme of Green and Tao, set out in \cite{greentao10:1}, that generalized their famous result about arithmetic progressions in the primes \cite{greentao08:1} to a very wide class of linear configurations, and gave the correct asymptotics for each one (which did not follow, even for arithmetic progressions, from their earlier work). 

\begin{theorem} \label{inversetheorem}
For every positive integer $s$ and every $\d>0$ there exists a finite collection $\mathcal M$ of $s$-step manifolds, each with a Riemannian metric, and positive constants $C$ and $c$ with the following property. For every $N\geq 1$ and every function $f:\{1,2,\dots,N\}\to\C$ such that $\|f\|_\infty\leq 1$ and $\|f\|_{U^{s+1}}\geq\d$ there is a nilmanifold $G/\Gamma$ in $\mathcal M$, an element $g\in\Gamma$, and a function $F:G/\Gamma\to\C$ such that $\|F\|_\infty\leq 1$, the Lipschitz constant of $F$ is at most $C$ (with respect to the given Riemannian metric), and
\[|\mathbb E_{n\leq N}f(n)\overline{F(g^nx)}|\geq c.\]
\end{theorem}

\noindent To put this less formally, if $f$ has a large $U^{s+1}$-norm, then it must correlate well with an $s$-step nilsequence, where the nilmanifold comes from some finite collection of nilmanifolds and the Lipschitz constant of the function $F$ is not too large with respect to some sensible metric. To put it even less formally, functions with a large $U^{s+1}$-norm correlate with $s$-step nilsequences.

An important remark is that the converse holds as well: if a function takes bounded values and correlates with an $s$-step nilsequence satisfying the above condition, then it has a large $U^{s+1}$-norm. This is a much easier result, though it is more than just a simple exercise: it was proved by Green and Tao in \cite{greentao08:2}. Thus, the inverse theorem really does \emph{characterize} functions with large $U^{s+1}$-norm. 

Right back at the beginning of Section \ref{quadfanal} I said that although quadratic Fourier analysis lacked an inversion formula, it had a weaker property that was adequate for several applications. That property is Theorem \ref{inversetheorem}, the inverse theorem, so in fact the remark applies to degree-$s$ Fourier analysis for all $s$. 

One way to see that an inverse theorem is sometimes enough is to look back at the proof of Roth's theorem. Although we used the inversion formula -- that is, the statement that a function can be uniquely decomposed as a linear combination of trigonometric functions -- all that we actually needed for the proof was to be able to show somehow that a bounded function with large $U^2$-norm correlated well with at least one trigonometric function. The structure of trigonometric functions was then enough to allow us to find increased density on a subprogression. The inverse theorem for the $U^{s+1}$-norm can be used in a similar (but more complicated) way to yield another proof of Szemer\'edi's theorem for progressions of length $s+2$: this was shown by Green and Tao in \cite{greentao10:2}.

Another reason that inverse theorems can be regarded as a substitute for Fourier analysis in this context is that they often lead to useful decomposition theorems. Roughly speaking, if an inverse theorem for a norm $\|.\|$ shows that every bounded function with a large norm must correlate with a function from a set $\mathcal F$, then one can deduce from it that every function can be written as a linear combination of elements of $\mathcal F$, with the absolute values of the coefficients having not too large a sum, plus a function with small norm, plus an ``error" function that typically has a small $L_p$ norm for some $p$ such as 1 or 2. One way of proving this, due to Green and Tao \cite{greentao08:2}, is modelled on arguments from ergodic theory: if $\|f\|$ is large one uses the inverse theorem to find a function $F\in\mathcal F$ that correlates well with $f$, defines a ``sigma-algebra" with respect to which $F$ is ``approximately measurable", and then repeats the process with $f-Pf$, continuing until the desired decomposition is achieved. Another approach uses the Hahn-Banach theorem to obtain a contradiction of the inverse theorem if the desired decomposition does not exist: see \cite{gowerswolf12,gowerswolf11:1} for some applications of this idea, and \cite{gowers10} for a general discussion of the method.

\subsection{What more could one ask for?} {\ }
\medskip

The inverse theorem of Green, Tao and Ziegler is a major highlight of additive combinatorics, with a major application to analytic number theory. It might seem a little greedy to ask for more, but there is nevertheless a feeling that the story of the inverse theorem is not yet finished. There are two main reasons for this. The first is that the proof yields no bound at all. It is likely that one could rewrite the proof to make it finitary, but the resulting bound would be very weak and would not justify the significant amount of work that would be needed to do this. (It may seem paradoxical that a non-quantitative result can be used to obtain asymptotics for configurations in the primes: the point is that these asymptotics are accurate to within a $1+o(1)$ factor and we know nothing about the rate of convergence of the $o(1)$ part.)

A second reason is that there is something rather non-canonical about the statement of the inverse theorem. As we have already remarked, conventional Fourier analysis gives a unique decomposition of a function as a linear combination of trigonometric functions, while for higher-degree Fourier analysis we do not have any notion of a ``Fourier transform". But even the statement of the inverse theorem, with a Lipschitz function that is often chosen rather arbitrarily, and a nilpotent Lie group that sometimes feels a little cooked up to yield something like a bracket polynomial, which itself is not a very natural object, does not give one the feeling that it is ``from the book". 

A programme of Szegedy aims to remedy the second of these ``defects" but not the first. In a series of papers \cite{szegedy10, camarenaszegedy12, szegedy12}, one joint with Omar Camarena, he works with abstract structures that he calls nilspaces, which are variants of abstract parallelepiped structures introduced by Host and Kra \cite{hostkra08} (abstracting out certain arguments from \cite{hostkra05}). Such structures can be thought of as the most general structures for which one can make sense of uniformity norms, and are therefore a natural setting for thinking about inverse theorems. (This point of view is one way of explaining why nilpotent groups enter the theory when the original questions are about Abelian groups. It turns out that a group does not have to be Abelian to be a suitable host for abstract parallelepiped structures.) 

Szegedy's approach is decidedly infinitary -- indeed, the thinking behind it is that when one passes to suitable limiting objects, much of the ``mess" that makes the theory difficult disappears. Unfortunately, it has proved to be extremely hard to check the correctness of the arguments in the three papers, which, if all the details can be completed and checked, would give a different and in some ways more natural proof of the inverse theorem. At the time of writing, various people are working to produce clearer and more complete versions of the argument \cite{candela16:1, candela16:2, gutmanmannersvarju16:1, gutmanmannersvarju16:2, gutmanmannersvarju16:3}: it seems likely that Szegedy's ideas are fundamentally correct and that this is indeed an interesting alternative approach.

The hope, however, would be that out there is a much simpler proof (and statement) of the inverse theorem that yields good bounds. Perhaps a Grothendieck-like figure will one day find the right abstract framework that will make the difficulties melt away. One can at least dream.

\section{Hypergraphs}

A graph is a collection of pairs of elements of a set. What happens if we generalize from pairs to triples and beyond? A $k$-\emph{uniform hypergraph} is a set $X$ and a subset of $X^{(k)}$, where $X^{(k)}$ denotes the set of all subsets of $X$ of size $k$. In this section I shall concentrate on the case $k=3$, though it should be fairly clear how to generalize what I say to higher values.

Just as it is natural, when one thinks about graphs in an analytic way, to think of them as special kinds of matrices, or functions of two variables, so hypergraphs can be thought of as functions of three variables. Furthermore, there is a natural three-variable analogue of the box norm that we saw earlier. It is given by the following formula.
\begin{align*}
\|f\|_{\square^3}^8=\E_{x,x',y,y',z,z'}f(x,y,z)&\overline{f(x,y,z')f(x,y',z)}f(x,y',z')\\
&\overline{f(x',y,z)}f(x',y,z')f(x',y',z)\overline{f(x',y',z')}.\\
\end{align*}
As usual, one can define a corresponding box inner product by using eight different functions instead of just one, the inner product satisfies a Cauchy-Schwarz-type inequality, and that inequality can be used to prove that the norm really is a norm. Now let us look at some further useful facts about the box norm.

There is enough similarity between the formula for the box norm and the formula for the $U^3$ norm for it to be highly plausible that there should be a close relationship between them. And indeed there is. Let $G$ be a finite Abelian group, let $f:G\to\C$ be some function, and define a three-variable function $F:G^3\to\C$ by $F(x,y,z)=f(x+y+z)$. It is easy to check directly from the formula that $\||F\|_{\square^3}=\|f\|_{U^3}$. (A similar relationship can also be shown between the two-dimensional box norm and the $U^2$ norm.) 

It is a little surprising, therefore, that one can prove rather easily an inverse theorem for the box norm. As we shall see, however, the information it gives us is not strong enough to allow us to deduce from it the inverse theorem for the $U^3$ norm.

Let $X$ be a finite set and let $f:X^3\to\C$ be a function with $\|f\|_\infty\leq 1$ and $\|f\|_{\square^3}\geq c$. The second inequality tells us that 
\begin{align*}
\E_{x,x',y,y',z,z'}f(x,y,z)\overline{f(x,y,z')f(x,y',z)}&f(x,y',z')\\
&\overline{f(x',y,z)}f(x',y,z')f(x',y',z)\overline{f(x',y',z')}\geq c^8.\\
\end{align*}
By averaging, there must exist $x',y',z'$ such that
\begin{align*}
|\E_{x,y,z}f(x,y,z)\overline{f(x,y,z')f(x,y',z)}&f(x,y',z')\\
&\overline{f(x',y,z)}f(x',y,z')f(x',y',z)\overline{f(x',y',z')}|\geq c^8.\\
\end{align*}
We can think of the left-hand side as the modulus of the inner product of $f$ with the function $g$, given by the formula
\[g(x,y,z)=f(x,y,z')f(x,y',z)\overline{f(x,y',z')}f(x',y,z)\overline{f(x',y,z')f(x',y',z)}f(x',y',z').\]
The interesting thing about $g$ is that it is a product of functions each of which depends on at most two of the variables $x,y,z$. Thus, we find that if $\|f\|_{\square^3}$ is large, then it correlates with a function of ``lower complexity". The analogue of these low-complexity functions for matrices is the matrices of the form $u\otimes v$ -- that is, the matrices of rank 1.

However, in the two-variable case we have more. For Hermitian matrices we have a decomposition of the form $\sum_i\lambda_iu_i\otimes\overline{u_i}$, where $(u_i)$ is an orthonormal basis, and in general we have a singular-value decomposition $\sum_i\lambda_iu_i\otimes\overline{v_i}$, where $(u_i)$ and $(v_i)$ are both orthonormal. If $u,v$ and $w$ are functions of two variables, let us write $[[u,v,w]]$ for the function whose value at $(x,y,z)$ is $u(x,y)v(y,z)w(z,x)$. Then the very simple inverse theorem just proved tells us that a function with large box norm correlates with a function of the form $[[u,v,w]]$, but what we do not seem to have is a canonical way of decomposing an arbitrary function as a sum of the form $\sum_i\lambda_i[[u_i,v_i,w_i]]$.

What happens if we try to deduce the inverse theorem for the $U^3$ norm from the inverse theorem for the box norm in three variables? If $\|f\|_{U^3}\geq c$, then the argument gives us functions $f_1,\dots,f_6$, all of $\ell_\infty$ norm at most 1, such that
\[|\E_{x,y,z}f(x+y+z)\overline{f_1(x+y)f_2(y+z)f_3(z+x)}f_4(x)f_5(y)f_6(z)|\geq c^8.\]
However, it does not tell us anything much about the structure of the functions $f_1\dots,f_6$. It is possible to deduce from the inequality above that they have quadratic structure, and that the inverse theorem therefore holds, but the proof is no easier than the proof of the inverse theorem was already -- it just uses the same general approach in an unnecessarily complicated way.

Despite this, the theory of hypergraphs has been important and useful in additive combinatorics. I will not explain why here, except to mention a theorem about hypergraphs that turns out to imply a multidimensional version of Szemer\'edi's theorem. It is known as the \emph{simplex removal lemma}. (The implication, observed by Solymosi \cite{solymosi04}, is fairly straightforward, but slightly too long to give here.) Define a \emph{simplex} in a $k$-uniform hypergraph $H$ to be a set of $k+1$ vertices such that any $k$ of them form an edge $H$. (The word ``edge" here means one of the sets of size $k$ that belongs to $H$. When $k=2$, a simplex is a triangle.) The following result is due to Nagle, R\"odl, Schacht and Skokan \cite{rodlskokan04, naglerodlschacht06}, and independently to the author \cite{gowers07}. (See also \cite{tao06}.)

\begin{theorem}
For every $c>0$ and positive integer $k$ there exists $a>0$ with the following property. If $H$ is a $k$-uniform hypergraph with $n$ vertices that contains at most $an^{k+1}$ simplices, then it is possible to remove at most $cn^k$ edges from $H$ to create a $k$-uniform hypergraph that contains no simplices at all.
\end{theorem}

The case $k=2$, which pioneered this combinatorial approach to Szemer\'edi's theorem, was proved by Ruzsa and Szemer\'edi much earlier \cite{ruzsaszem78}. In this case the result says that a graph with few triangles is close to a graph with no triangles. Rather surprisingly, even this case is not straightforward. In particular, the best known dependence of $a$ on $c$ is extremely weak: its reciprocal is a tower of 2s of height proportional to $\log(1/c)$ \cite{fox11}. Even a bound of the form $\exp(-(1/c)^A)$ for some fixed $A>0$ would be a major improvement.

\section{Fourier analysis on non-Abelian groups}

The following result is easy to prove. We say that a subset of an Abelian group is \emph{sum free} if it contains no three elements $x,y,z$ with $x+y=z$.

\begin{theorem}
There exists a constant $c>0$ such that every finite Abelian group $G$ has a subset $A$ of cardinality at least $c|G|$ that is sum free.
\end{theorem}

To see this, let $\Z_m$ be one of the cyclic groups of which $G$ is a product, and take all elements whose coordinate in this copy of $\Z_m$ lies between $m/3$ and $2m/3$ (and strictly between on one of the two sides). 

Babai and S\'os asked whether a similar result held for general finite groups \cite{babaisos85}. They expected the answer no, but it turns out not to be completely obvious how to disprove it.

Given that the result holds for Abelian groups, it is natural to look at groups that are ``highly non-Abelian". This can be measured in various ways. One is to look at the sizes of conjugacy classes. If a group $G$ is Abelian, then all its conjugacy classes are singletons, so if a group has large conjugacy classes, then that is saying that in some sense it is far from Abelian: not only are the conjugates $gxg^{-1}$ not all equal to $x$, they are not even concentrated in a small subset of the group.

Another property that characterizes Abelian groups is that all their irreducible representations are one-dimensional. So another potential way of measuring non-Abelianness is to look at the lowest dimension of an irreducible representation. 

Since we have already made use of characters of finite Abelian groups -- that is, their irreducible representations -- and since we are trying to count solutions to a simple equation in a dense subset of a group, the second measure looks promising. And it does indeed turn out to be possible to solve this problem by using a more general Fourier analysis, in which characters are replaced by more general irreducible representations. 

The definition of the Fourier transform of a function $f:G\to\C$ is more or less the first thing one writes down. If $\rho:G\to U(k)$ is an irreducible unitary representation of $G$, then 
\[\hat f(\rho)=\E_xf(x)\overline{\rho(x)}.\]
(Another candidate for the definition would be as above but with the conjugate $\overline{\rho(x)}$ replaced by the adjoint $\rho(x)^*$, but the conjugate turns out to be more convenient.)

For this to be a useful definition, we would like it to satisfy natural analogues of the basic properties of the Abelian Fourier transform. And indeed it does. Parseval's identity, for example, takes the following form. If $f$ and $g$ are functions from $G$ to $\C$, then
\[\E_x f(x)\overline{g(x)}=\sum_\rho n_\rho\tr(\hat f(\rho)\hat g(\rho)^*),\]
where the sum is over all irreducible representations and for each such representation $\rho$ its dimension is $n_\rho$. Let us briefly see how this is proved. We have
\begin{align*}
\sum_\rho n_\rho\tr(\hat f(\rho)\hat g(\rho)^*)&=\sum_\rho n_\rho\E_{x,y}f(x)\overline{g(y)}\overline{\tr(\rho(x)\rho(y)^*)}\\
&=\E_{x,y}f(x)\overline{g(y)}\sum_\rho n_\rho\overline{\tr(\rho(x)\rho(y)^*)}\\
\end{align*}
We now use a fundamental orthogonality result from basic representation theory, which states that $\sum_\rho n_\rho\tr(\rho(x)\rho(y)^*)=n$ if $x=y$ and 0 otherwise. It follows that 
\[\E_{x,y}f(x)\overline{g(y)}\sum_\rho n_\rho\overline{\tr(\rho(x)\rho(y)^*)}=\E_xf(x)\overline{g(x)}\]
and the proof is complete.

How about the convolution identity? It states, as we would hope, that
\[\widehat{f*g}(\rho)=\hat f(\rho)\hat g(\rho)\]
for any two functions $f,g:G\to\C$ and any irreducible representation $\rho$. Again it is instructive to see the proof. We have
\begin{align*}
\widehat{f*g}(\rho)&=\E_x(f*g)(x)\overline{\rho(x)}\\
&=\E_x\E_{uv=x}f(u)g(v)\overline{\rho(x)}\\
&=\E_{u,v}f(u)g(v)\overline{\rho(u)\rho(v)}\\
&=(\E_uf(u)\overline{\rho(u)})(\E_vg(v)\overline{g(v)})\\
&=\hat f(\rho)\hat g(\rho).\\
\end{align*}
Note that we used the fact that $\overline{\rho(uv)}=\overline{\rho(u)\rho(v)}$ in the proof above. Had we defined the Fourier transform using adjoints, we would have had to use instead the fact that $\rho(uv)^*=\rho(v)^*\rho(u)^*$, so we would have obtained the identity $\widehat{f*g}(\rho)=\hat g(\rho)\hat f(\rho)$.

The last property I want to discuss is the inversion formula. Here we have what looks at first like a puzzle: in the Abelian case we decomposed functions as linear combinations of characters, but irreducible representations are matrix-valued functions of different dimensions, so we cannot express scalar-valued functions as linear combinations of them.

There is of course a natural way of converting a matrix-valued function into a scalar-valued function, and that is to take the trace. Moreover, traces of representations are well known to be important functions -- they are characters in the sense of representation theory. 

So can we decompose a function as a linear combination of functions of the form $\chi(x)=\tr(\rho(x))$? No we cannot, since such functions are constant on conjugacy classes. (We can, however, decompose functions if they \emph{are} constant on conjugacy classes -- such functions are called \emph{class functions}.) In fact, since there are not $n$ inequivalent irreducible representations (except when the group is Abelian), there is no hope of writing down some scalar-valued functions $u_\rho$ and expanding every $f$ as a linear combination of the $u_\rho$.

However, we shouldn't necessarily expect to be able to do so. We would like the coefficients in our inversion formula to be the matrices $\hat f(\rho)$ in some suitable sense. And once we make that our aim, it is a short step to writing down the following slightly subtler formula.
\[f(x)=\sum_\rho n_\rho\tr(\hat f(\rho)\overline{\rho(x)^*}).\]
This can be verified easily using the orthogonality property we used earlier.

It is not hard to check that this formula specializes to the formula given earlier when the group is Abelian. One way of making it look more like that formula is to define $\hat G$ to be the set of all irreducible representations of $G$ (up to equivalence), to define $M(\hat G)$ to be the set of all matrix-valued functions $\hat f$ on $\hat G$ such that $\hat f(\rho)$ is an $n_\rho\times n_\rho$ matrix for every $\rho$, and to define an inner product on $M(\hat G)$ by the formula
\[\langle \hat f,\hat g\rangle = \sum_\rho n_\rho\langle\hat f(\rho),\hat g(\rho)\rangle,\]
where the inner product on the right-hand side is the matrix inner product $\langle A,B\rangle=\tr(AB^*)=\sum_{i,j}A_{ij}\overline{B_{ij}}$. Note that Parseval's identity now becomes the usual formula $\langle f,g\rangle=\langle\hat f,\hat g\rangle$. As for the inversion formula, it can be written as follows.
\[f(x)=\langle \hat f,\d_x^*\rangle,\]
where $\d_x^*$ is the evaluation function $\rho\mapsto\rho(x)^*$. The right-hand side can be expanded to $\sum_\rho n_\rho\langle\hat f(\rho),\overline{\rho(x)}\rangle$, which is equal to $\sum_\chi\hat f(\chi)\chi(x)$ when $G$ is Abelian.

Now let us prove an inequality, Lemma \ref{bnp} below, that allows us to solve the problem with which we started. It appears in equivalent form as Lemma 3.2 in \cite{gowers08} (strictly speaking that lemma is very slightly less general, but it is sufficient for applications to characteristic functions of sets and the proof carries through with hardly any changes for general functions). The formulation below is due to Babai, Nikolov and Pyber \cite{babainikolovpyber08}, who gave a different argument. Here we give a short Fourier-analytic argument that is different again. After the appearance of the paper \cite{gowers08}, the existence of such an approach seems to have been realized by various people and become a piece of modern folklore: I heard that it could be done by Ben Green, and Terence Tao gives it as an exercise in a blog post on non-Abelian Fourier analysis (\url{https://terrytao.wordpress.com/2011/12/16/254b-notes-3-quasirandom-groups-expansion-and-selbergs-316-theorem/}).

Before we give the lemma and explain how the problem of Babai and S\'os can be solved, we need a couple of simple results about matrices. Note that because we are on the Fourier side, we are dealing with sums rather than expectations. In particular, we are using the standard notion of matrix multiplication, and the box norm will be defined using sums.

\begin{lemma}
Let $A$ and $B$ be square matrices. Then $\|AB\|_{HS}\leq\|A\|_\square\|B\|_\square$.
\end{lemma}

\begin{proof} Observe that 
\begin{align*}
\|AB\|_{HS}^2&=\sum_{x,x'}|\sum_yA(x,y)B(y,x')|^2\\
&=\sum_{x,x'}\sum_{y,y'} A(x,y)\overline{A(x,y')B^*(x',y)}B^*(x',y').\\
\end{align*}
This last expression is the ``box inner product" $[A,A,B^*,B^*]$, and as we saw earlier (in the section on matrices) it satisfies a Cauchy-Schwarz-type inequality
\[[A,B,C,D]\leq\|A\|_\square\|B\|_\square\|C\|_\square\|D\|_\square.\]
Applying this, together with the fact that $\|B^*\|_\square=\|B\|_\square$, we obtain the result.
\end{proof}

\begin{lemma}
For every matrix $A$ we have $\|A\|_\square\leq\|A\|_{HS}$.
\end{lemma}

\begin{proof}
This can be shown with a direct argument, but it also follows from the fact that $\|A\|_\square$ is the $\ell_4$ norm of the singular values of $A$ and $\|A\|_{HS}$ is the $\ell_2$ norm.
\end{proof}

\begin{lemma} \label{bnp}
Let $G$ be a finite group and let $f,g:G\to\C$ be functions with average zero. Let $m$ be the smallest dimension of a non-trivial representation of $G$. Then 
\[\|f*g\|_2\leq m^{-1/2}\|f\|_2\|g\|_2.\]
\end{lemma}

\begin{proof}
By the convolution identity, Parseval's identity and the lemmas above, we have that
\begin{align*}
\|f*g\|_2^2&=\sum_\rho n_\rho\|\hat f\hat g\|_{HS}^2\\
&\leq\sum_\rho n_\rho\|\hat f\|_\square^2\|\hat g\|_\square^2\\
&\leq\sum_\rho n_\rho\|\hat f\|_{HS}^2\|\hat g\|_{HS}^2.\\
\end{align*}
Since $f$ averages zero, $\hat f(\rho)=0$ when $\rho$ is the trivial representation. Also, by Parseval's identity we have that $\sum_\rho n_\rho\|\hat f(\rho)\|_{HS}^2=\|f\|_2^2$. It follows that the maximum possible value of $\|\hat f(\rho)\|_{HS}^2$ is $m^{-1}\|f\|_2^2$. Therefore, using Parseval's identity again, we find that
\[\sum_\rho n_\rho\|\hat f\|_{HS}^2\|\hat g\|_{HS}^2\leq m^{-1}\|f\|_2^2\sum_\rho n_\rho\|\hat g(\rho)\|_{HS}^2=m^{-1}\|f\|_2^2\|g\|_2^2,\]
which completes the proof.
\end{proof}

Now let us quickly deduce that if a group $G$ has no non-trivial low-dimensional representations, then it does not contain a large product-free set. 

\begin{theorem}
Let $G$ be a finite group and let $m$ be the smallest dimension of a non-trivial representation of $G$. Then $G$ contains no product-free subset of density greater than $m^{-1/3}$. 
\end{theorem}

\begin{proof}
Let $\a$ be the density of $A$ and as usual let $f$ be the function $f(x)=A(x)-\a$. We shall now try to show that
\[\E_{xy=z}A(x)A(y)A(z)\ne 0,\]
which obviously implies that $A$ is not product free.

We have that
\[\E_{xy=z}A(x)A(y)A(z)=\E_{xy=z}(\a+f(x))(\a+f(y))(\a+f(z)),\]
and since $f$ averages zero, if we expand the right-hand side into eight separate sums, we find that all terms are zero apart from two, and we obtain the expression
\[\a^3+\E_{xy=z}f(x)f(y)f(z)=\a^3+\langle f*f,f\rangle.\]
By the Cauchy-Schwarz inequality and Lemma \ref{bnp} we have that
\[|\langle f*f,f\rangle|\leq\|f*f\|_2\|f\|_2\leq m^{-1/2}\|f\|_2^3.\]
We also have that $\|f\|_2^2=\a(1-\a)^2+(1-\a)\a^2=\a(1-\a)$. Therefore, if $A$ is product free we must have the inequality
\[\a^{3/2}(1-\a)^{3/2}m^{-1/2}\geq\a^3,\]
which implies that $\a\leq m^{-1/3}$.
\end{proof}

It remains to remark that there do exist groups with no low-dimensional representations. Indeed, any family of finite simple groups has this property, though some have it much more strongly than others. The ``most" non-Abelian family of groups is the family PSL$(2,q)$. If a group in this family has order $n$, then its non-trivial representations have dimension at least $cn^{1/3}$, where $c>0$ is an absolute constant. Therefore, these groups have no product-free subsets of density greater than $c'n^{-1/9}$.

The same argument shows that if $A$, $B$ and $C$ are sets of density greater than $m^{-1/3}$, then $ABC=\{abc:a\in A, b\in B, c\in C\}=G$. 

There turns out to be a close connection between groups with no low-dimensional representations and quasirandom graphs. If $G$ is a finite group with no low-dimensional non-trivial representations, then for any dense set $A\subset G$ we can define a bipartite graph with two copies of $G$ as its vertex sets and $x$ joined to $y$ if and only if $y=ax$ for some $a\in A$. The remark above about sets $A,B$ and $C$ tells us that this graph is quasirandom. 

As a final remark, we note that the $U^2$ norm can be generalized easily to a non-Abelian context. A good definition turns out to be as follows.
\[\|f\|_{U^2}^4=\E_{xy^{-1}zw^{-1}=e}f(x)\overline{f(y)}f(z)\overline{f(w)}.\]
The properties of this norm are just what one would hope. For example, one can define a generalized inner product in the obvious way, and we do indeed have the inequality
\[[f_1,f_2,f_3,f_4]\leq\|f_1\|_{U^2}\|f_2\|_{U^2}\|f_3\|_{U^2}\|f_4\|_{U^2},\]
which can then be used in the usual way to prove that this $U^2$ norm is a norm. We also have the Fourier interpretation that one would guess, namely
\[\|f\|_{U^2}^4=\sum_\rho n_\rho\|\hat f(\rho)\|_\square^4.\]
This is the natural guess because it involves fourth powers on the right-hand side, both in the obvious sense that there is a fourth power visible in the expression, and also in the less obvious sense that the box norm of a matrix is equal to the $\ell_4$ norm of the singular values. (Thus, in a certain sense we have fourth powers of generalized Fourier coefficients in two different ways.) Indeed, there is a natural way of defining an $\ell_p$ norm on $\hat G$ for every $p$. For an $m\times m$ matrix $A$, one defines the \emph{trace-class norm} $\|A\|_p$ to be the $\ell_p$ norm of the singular values of $A$, and then for a matrix-valued function $\hat f$ one defines $\|f\|_p$ by the formula
\[\|\hat f\|_p^p=\sum_\rho n_\rho\|\hat f(\rho)\|_p^p.\]
That is, we take the $\ell_p$ norm in $\hat G$ of the function $x\mapsto\|\hat f(x)\|_p$. Once we have done this, we have the familiar identity
\[\|f\|_{U^2}=\|\hat f\|_4.\]

Functions with small $U^2$ norms (given their averages) behave like random functions, and when a group has no non-trivial low-dimensional representations, Lemma \ref{bnp} tells us that \emph{all} reasonably spread out functions behave like random functions. To see this, note that $\|f\|_{U^2}^4=\|f*f^*\|_2^2$, where $f^*(x)$ is defined to be $\overline{f(x^{-1})}$, so if $f$ averages zero, then we have the inequality $\|f\|_{U^2}^4\leq m^{-1}\|f\|_2^2$.

\section{Fourier analysis for matrix-valued functions}

Let $G$ be a finite group and let $f:G\to$M$_n(\C)$ be a matrix-valued function. (We are not assuming that $n$ is the order of $G$.) We can define a Fourier transform for $f$ by simply applying the definition of the previous section to each matrix coefficient. That is, for each $i,j\leq n$ we define $f_{ij}$ to be the function $x\mapsto f(x)_{ij}$, and then for each irreducible representation $\rho$ we define $\hat f(\rho)$ to be the $n\times n$ block matrix whose $ij$th entry is the $n_\rho\times n_\rho$ matrix $\widehat {f_{ij}}(\rho)$. Thus, $\hat f(\rho)$ is an $nn_\rho\times nn_\rho$ matrix. 

We can write this definition more concisely, and in a basis-free way, as follows.
\[\hat f(\rho)=\E_xf(x)\otimes\overline{\rho(x)}.\]
Thus, whereas with Abelian groups we had scalar-valued functions and scalar-valued representations, and in the previous section we had scalar-valued funtions and matrix-valued representations, now we have matrix-valued functions and matrix-valued representations. In each case we take tensor products, but in the first two cases they are trivial.

In order to state the basic properties of the Fourier transform, we need to be clear about our notation. For a matrix-valued function $f$ on the physical side, we shall write $\|f\|_2$ for the norm defined by the formula
\[\|f\|_2^2=\E_x\|f(x)\|_2^2\]
Here it turns out to be convenient to take $\|f(x)\|_2$ to be the non-normalized Hilbert-Schmidt norm. Thus, the norm scales with the dimension of $f$, but not with the size of the group. 

On the Fourier side, we have similar definitions but using sums all the way through, so these are the same as the definitions of the norms and inner product in the scalar case.

With these normalizations, the first few basic properties of the Fourier transform now read as follows.
\begin{itemize}
\item $\|f\|_2^2=\|\hat f\|_2$. (Parseval's identity).
\item $\langle f,g\rangle=\langle\hat f,\hat g\rangle$. (Parseval's identity).
\item $\widehat{f*g}(\rho)=\hat f(\rho)\hat g(\rho)$. (Convolution formula).
\end{itemize}
The proofs are more or less the same as in the scalar case, but to clarify the point about normalizations, we give a proof of the second version of Parseval's identity, which goes like this. (It is important to realize that the meaning of the inner product varies from expression to expression -- sometimes we are talking about the inner product of two matrices, and sometimes about the inner product of two matrix-valued functions.)
\begin{align*}
\langle\hat f,\hat g\rangle &= \sum_\rho n_\rho \langle\hat f(\rho),\hat g(\rho)\rangle\\
&=\sum_\rho n_\rho\langle\E_x f(x)\otimes\overline{\rho(x)},\E_y g(y)\otimes\overline{\rho(y)}\rangle\\
&=\E_{x,y}\langle f(x),g(y)\rangle\sum_\rho n_\rho\overline{\langle\rho(x),\rho(y)\rangle}\\
\end{align*}
In the last expression, both inner products use sums. 

By the basic orthogonality property from representation theory, the sum over $\rho$ is equal to $|G|\d_{xy}$, so we end up with $\E_x\langle f(x),g(x)\rangle$, which is the definition of $\langle f,g\rangle$. 

The inversion formula is also straightforward, but it needs a little notation. Recall that for scalar-valued functions the inversion formula was $f(x)=\sum_\rho n_\rho\tr(\hat f(\rho)\overline{\rho(x)^*})$. Since the Fourier transform for matrix-valued functions is obtained by applying the Fourier transform for scalar-valued functions to each matrix entry, we obtain the formula
\[f(x)_{ij}=\sum_\rho n_\rho\tr(\widehat{f_{ij}}(\rho)\overline{\rho(x)^*}).\]
Let us write $\tr_\rho$ for the operation that takes an $n\times n$ block matrix $A$ with blocks that are $n_\rho\times n_\rho$ matrices and returns the $n\times n$ matrix whose $ij$th value is the (unnormalized) trace of the $ij$th block of $A$. Then we can write the inversion formula in the form
\[f(x)=\sum_\rho\tr_\rho(\hat f(\rho)\overline{\rho(x)^*}),\]
which is just like the formula when $f$ takes scalar values except that the trace function $\tr$ has been replaced by the matrix-of-traces function $\tr_\rho$.

This matrix-valued Fourier transform was introduced, with slightly different conventions, by Moore and Russell \cite{moorerussell15} (not the famous philosophers, but a pair of contemporary mathematicians). It is useful when one wishes to measure the extent to which a matrix-valued function behaves like a representation. To illustrate this, let us look at a very nice result that shows in a simple way how the matrix-valued transform can be used.

In the previous section, we remarked that if $G$ is a finite group with no low-dimensional non-trivial representations, then for every function $f:G\to\C$ such that $\|f\|_\infty\leq 1$ and $\E_xf(x)=0$ the $U^2$ norm of $f$ is small. That is,
\[\E_{xy^{-1}zw^{-1}}f(x)\overline{f(y)}f(z)\overline{f(w)}\]
is small. Now if we could find a non-trivial character of $G$, in the sense of a homomorphism from $G$ to $\C$ that is not the identity, then this would not be true: whenever $xy^{-1}zw^{-1}=e$ we would have $f(x)\overline{f(y)}f(z)\overline{f(w)}=1$ so the average would be 1, which is the largest it can possibly be. So the observation that the $U^2$ norm has to be small is telling us that $G$ not only fails to have a non-trivial character (which we know because it has no non-trivial low-dimensional representations), but it does not admit any functions that are even very slightly \emph{close} to being a non-trivial character: if $f$ averages zero, then it not possible for the average real part of $f(x)\overline{f(y)}f(z)\overline{f(w)}$ to be greater than some small constant when $xy^{-1}zw^{-1}=e$.

Moore and Russell showed that this observation can be extended to matrix-valued functions. (Actually, the precise result they showed was not quite this one, but it was very similar and had a very similar proof.) That is, if the dimension of the smallest non-trivial representation is $m$, if $n$ is substantially less than $m$, and if $f$ takes values that are $n\times n$ matrices with operator norm at most 1, and if $\E_xf(x)=0$, then the $U^2$ norm of $f$ is small. This is a stronger statement, because matrix-valued functions have more elbow room and therefore more room to create the necessary correlations. It is also stronger in a more obvious way: it tells us that not only are scalar-valued functions on $G$ as unlike non-trivial characters as they could possibly be, low-dimensional matrix-valued functions are as far from non-trivial representations as they could possibly be.

We begin by observing that the statement and proof of Lemma \ref{bnp} carry over almost word for word to the matrix-valued case. (The proof is so close that we do not give it again.) 

\begin{lemma} \label{matrixbnp}
Let $G$ be a finite group and let $m$ be the smallest dimension of a non-trivial representation of $G$. Let $f,g:G\to$M$_n(\C)$ be two matrix-valued functions that average zero. Then
\[\|f*g\|_2\leq m^{-1/2}\|f\|_2\|g\|_2.\]
\end{lemma}

There is, however, an important new factor to take into account here, which is that the two sides of the inequality scale differently with the dimension. Suppose, for instance, that both $f$ and $g$ are equal to the same $n$-dimensional representation $\rho$. Then the 2-norms of every single $f(x)$, $g(x)$ and $f*g(x)$ are all equal to $n$, so $\|f*g\|_2^2=n$, while $m^{-1}\|f\|_2^2\|g\|_2^2=m^{-1}n^2$. So the inequality does not stop $f$ and $g$ from being representations when $n=m$, which of course is as it should be, since $m$ is defined to be the dimension of a representation of $G$.

With that remark in mind, let us turn to $U^2$ norms. The natural definition of the $U^2$ norm in the matrix-valued case is
\[\|f\|_{U^2}^4=\E_{xy^{-1}zw^{-1}=e}\tr(f(x)f(y)^*f(z)f(w)^*),\]
where the trace is not normalized. This is equal to $\|f*f^*\|_2^2$ and it is also equal to $\|\hat f\|_4^4$. (Recall that this is defined to be $\sum_\rho n_\rho\|\hat f\|_\square^4$.) Therefore, by Lemma \ref{matrixbnp} we find that
\[\|f\|_{U^2}^4\leq m^{-1}\|f\|_2^4.\]
In particular, if $f$ takes values $f(x)$ with Hlibert-Schmidt norm at most $n^{1/2}$ (which is the case, for example, if they are all unitary, and more generally if they all have operator norm at most 1), then $\|f\|_{U^2}^4\leq m^{-1}n^2$. For $f$ to be a unitary representation, we would need $\|f\|_{U^2}^4$ to be equal to $n$ (since $f(x)f(y)^*f(z)f(w)^*$ would be the identity whenever $xy^{-1}zw^{-1}=e$), which would equal $n^{-1}\|f\|_2^4$, which is at most $n$, by hypothesis. Therefore, if $n$ is significantly less than $m$, so that $m^{-1}n^2$ is significantly less than $n$, we see that $f$ is not even close to being a representation, in the sense that there is almost no correlation between $f(x)f(y)^*$ and $f(w)f(z)^*$ even if we are given that $xy^{-1}=wz^{-1}$. 

\section{An inverse theorem for the matrix $U^2$ norm}

Recall the very simple inequalities that we used earlier to relate the $\ell_4$ and $\ell_\infty$ norms of the Fourier transform of a function $f$ from an Abelian group to $\C$. If we know that $\|f\|_2\leq 1$, then we find that
\[\|\hat f\|_\infty^4\leq\|\hat f\|_4^4\leq\|\hat f\|_\infty^2\|\hat f\|_2^2\leq\|\hat f\|_\infty^2.\]
Since $\|\hat f\|_4=\|\hat f\|_{U^2}$, we deduce that if $\|f\|_{U^2}\geq c$, then there exists a character $\chi$ such that $|\langle f,\chi\rangle|=|\hat f(\chi)|\geq c^2$.

What happens if we try to generalize this to non-Abelian groups and to matrix-valued functions? Let us assume that $f(x)$ is an $n\times n$ matrix with operator norm at most 1 for every $x$. (The operator norm is the maximum of the singular values, and thus the natural $\ell_\infty$ norm of a matrix.) Then just as before, we have
\[\|\hat f\|_\infty^4\leq\|\hat f\|_4^4=\sum_\rho n_\rho\|\hat f(\rho)\|_4^4\leq\max_\rho\|\hat f(\rho)\|_\infty^2\sum_\rho n_\rho\|\hat f(\rho)\|_2^2=\|\hat f\|_\infty^2\|f\|_2^2.\]
Unfortunately, if $n$ is large, then this is no longer a rough equivalence, since the best we can say about $\|f\|_2^2$ is that it is at most $n$ (since it is the sum of the squares of $n$ singular values, each of which lies between 0 and 1). 

However, that does not mean that there is nothing we can say. The largest possible value of $\|\hat f\|_4^4=\|f\|_{U^2}^4$ is, as we have seen, $n$. Let us take a function $f$ such that $\|f\|_{U^2}^4\geq cn$. Then we obtain from the second inequality above that there exists an irreducible representation $\rho$ such that $\|\hat f(\rho)\|_\infty\geq c^{1/2}$. 

If $G$ is Abelian and $f:G\to\C$, then $\|\hat f(\chi)\|_\infty$ is just $|\langle f,\chi\rangle|$, so this statement is saying that $f$ correlates in a significant way with a character. But in our situation we have the more complicated statement that
\[\|\E_xf(x)\otimes\overline{\rho(x)}\|_\infty\geq c^{1/2}.\]
Is this telling us that $f$ correlates in some sense with $\rho$?

Let us try to interpret it. We shall first use the fact that $\|A\|_\infty$, the operator norm of $A$, is the largest possible value of $\|Au\|_2$ over all unit vectors $u$, which in turn is the largest possible value of $\langle Au,v\rangle$ over all pairs of unit vectors $u$ and $v$. Therefore, we can find unit vectors $u$ and $v$ such that
\[\langle (\E_xf(x)\otimes\overline{\rho(x)})u,v\rangle\geq c^{1/2}.\]
Now let us rewrite this in coordinate form. Because of the special form of the $n\times n_\rho$ matrix $f(x)\otimes\overline{\rho(x)}$ it is natural to give it four indices instead of two: we have that $(f(x)\otimes\overline{\rho(x)})_{ijkl}=f(x)_{ik}\overline{\rho(x)_{jl}}$. Then indexing $u$ and $v$ in the corresponding way, and writing them as $U$ and $V$ since they have now become matrices, we have that
\[\langle (f(x)\otimes\overline{\rho(x)})U,V\rangle=\sum_{i,j,k,l}f(x)_{ik}\overline{\rho(x)_{jl}}U_{kl}\overline{V_{ij}}=\langle f(x)U\rho(x)^*,V\rangle,\]
where the product in the last expression is just normal matrix multiplication. Taking expectations, we deduce that
\[\E_x\langle f(x)U\rho(x)^*,V\rangle\geq c^{1/2}\]
for two $n\times n_\rho$ matrices $U$ and $V$ that have Hilbert-Schmidt norm 1. We can write this more symmetrically as
\[\E_x\langle f(x)U,V\rho(x)\rangle\geq c^{1/2}.\]
It is natural to rescale $U$ and $V$ so that they have Hilbert-Schmidt norm $n_\rho^{1/2}$. That is, we can say that there exist an irreducible representation $\rho$ and matrices $U$ and $V$ with $\|U\|_2^2=\|V\|_2^2=n_\rho$ such that
\[\E_x\langle f(x)U\rho(x)^*,V\rangle=\E_x\langle f(x)U,V\rho(x)\rangle=\E_x\langle f(x),V\rho(x)U^*\rangle\geq c^{1/2}n_\rho.\]

This seems quite satisfactory, but it falls short of being a true inverse theorem for the matrix $U^2$ norm because the converse does not hold. That is, if we are given $\rho,U$ and $V$ satisfying the above conditions, we cannot deduce that $f$ has a large $U^2$ norm unless $n_\rho$ is comparable to $n$, which it does not have to be.

Thus, we are in an interesting situation. Earlier, we thought of inverse theorems as something one settles for when one does not have an inversion formula. But here we have a clean and easily proved inversion formula that does not directly yield an inverse theorem. 

However, we have not yet exhausted our options. If $\|f\|_{U^2}^4=\|\hat f\|_4^4\geq cn$, then we are given that
\[\sum_\rho n_\rho\|\hat f(\rho)\|_4^4\geq cn,\]
where $\|\hat f(\rho)\|_4^4$ denotes the sum of the fourth powers of the singular values of $\hat f(\rho)$. Let these singular values be $\lambda_{\rho,i}$ for $i=1,2,\dots,n_\rho$. Then we find that
\[\sum_\rho n_\rho\sum_{i=1}^{n_\rho}\lambda_{\rho,i}^4\geq cn.\]
From Parseval's inequality and the assumption that each $f(x)$ has operator norm at most 1 (and hence Hilbert-Schmidt norm at most $n$) we also have that
\[\sum_\rho n_\rho\sum_{i=1}^{n_\rho}\lambda_{\rho,i}^2\leq n.\]
Also, since $\hat f(\rho)=\E_xf(x)\otimes\overline{\rho(x)}$ is an average of matrices with operator norm at most 1, every $\lambda_{\rho,i}$ is at most 1. 

Let $\lambda_1,\dots,\lambda_m$ be the singular values $\lambda_{\rho,i}$ arranged in some order, and for each $i$ let $n_i$ be the $n_\rho$ that corresponds to $\lambda_i$. Then we can rewrite these inequalities as
\[\sum_in_i\lambda_i^4\geq cn,\]
\[\sum_in_i\lambda_i^2\leq n,\]
and 
\[\lambda_i\leq 1.\]
Note that if $c=1$, then the only way of achieving the above inequalities is for $\lambda_i^4$ to equal $\lambda_i^2$ for every $i$ (assuming that none of the $n_i$ is zero). Thus, each $\lambda_i$ is either 0 or 1, and $\sum\{n_i:\lambda_i=1\}=n$. It is a straightforward exercise to prove that the more relaxed assumptions above lead to similar but more relaxed conclusion: we can find a set $A$ and constants $c_1>0$ and $C$ that depend on $c$ only (with a power dependence) such that $c_1n\leq\sum_{i\in A}n_i\leq Cn$, and $\lambda_i\geq c_1$ for every $i\in A$. In short, we can find a set of large singular values (coming from the various $\hat f(\rho)$ of size roughly comparable to $n$. 

With each such singular value $\lambda_i$ we can associate $n\times n_i$ matrices $U_i$ and $V_i$ with Hilbert-Schmidt norm $n_i$ such that $\E_x\langle f(x)U_i, V_i\rho_i\rangle\geq\lambda_in_i$, where $\rho_i$ is the representation corresponding to $\lambda_i$. Moreover, if two pairs $(U_i,V_i)$ and $(U_j,V_j)$ come from the same $\rho$, then because of the nature of singular value decompositions, we have that $\langle U_i,U_j\rangle=\langle V_i,V_j\rangle=0$.

It is plausible that we can put together these matrices and irreducible representations to create a representation-like function that correlates with $f$, and moreover that gives us an inverse theorem in the sense that the correlation in its turn implies that $f$ has a large $U^2$ norm. Exactly how the putting together should work is not obvious, but it turns out that it can be done. It yields the following theorem, due to the author and Omid Hatami \cite{gowershatami15}. In the statement, recall that $\|f\|_\infty$ means the largest operator norm of any $f(x)$. Also, we define a \emph{partial unitary} matrix to be an $n\times m$ matrix such that the rows are orthonormal if $n\leq m$ and the columns are orthonormal if $n\geq m$. (In particular, if $n=m$ then the matrix is unitary.) 

\begin{theorem} \label{matrixinverse}
Let $G$ be a finite group, let $c>0$ and let $f:G\to$M$_n(\C)$ be a function such that $\|f\|_\infty\leq 1$ and $\|f\|_{U^2}^4\geq cn$. Then there exists $m$ such that $cn/4\leq m\leq 4n/c$, an $m$-dimensional representation $\sigma$, and $n\times m$ partial unitary matrices $U$ and $V$, such 
that
\[\E_x\langle f(x)U,V\sigma(x)\rangle\geq c^2m/16.\]
\end{theorem}

Note that $\sigma$ will \emph{not} normally be irreducible. This theorem tells us that $f$ correlates with the function $V\sigma U^*$. The extra strength of this theorem over what we remarked earlier is that the dimension of $\sigma$ is comparable to that of $f$. It turns out to be simple to deduce a converse statement -- i.e., that if $f$ correlates with a function of the above form, then $\|f\|_{U^2}^4\geq c'n$ for some suitable $c'$. Therefore, Theorem \ref{matrixinverse} is indeed an inverse theorem for the matrix $U^2$ norm.

It is possible to give a more careful argument when $c=1-\e$ for some small $\e$ that allows us to show that $(1-2\e)n\leq m\leq(1-4\e)^{-1}n$, and to obtain a lower bound of $(1-16\e)m$ in the last inequality. From this result it is not too hard to deduce a so-called \emph{stability theorem} for near representations. Roughly speaking, it states that any unitary-valued function that almost obeys the condition to be a representation is close to a unitary representation.

\begin{theorem}
Let $G$ be a finite group and let $f:G\to U(n)$ be a function such that $\|f(x)f(y)-f(xy)\|_{HS}\leq\e\sqrt n$ for every $x,y\in G$. Then there exist $m$ with $(1-\e^2)n\leq m\leq (1-2\e^2)^{-1}n$, an $n\times m$ partial unitary matrix $U$, and a unitary representation $\rho:G\to U(m)$, such that
\[\|f(x)-U\rho(x)U^*\|\leq 31\e\sqrt n\]
for every $x\in G$.
\end{theorem}

When $\e$ is bounded above by $cn^{-1/2}$ for a suitable constant $c$, the inequality for $m$ forces $m$ to equal $n$. In this regime the result was known and is due to Grove, Karcher and Ruh \cite{grovekarcherruh74:1}. They also proved a stability result, this time with no restriction on $\e$, with the operator norm replacing the normalized Hilbert-Schmidt norm \cite{grovekarcherruh74:2} (see also \cite{kazhdan82}).

\section{Conclusion}

Now that we have seen several different generalizations of Fourier analysis (though not a complete list), we can draw up a checklist of the properties that a generalization is likely to need in order to be useful. Ideally we would have all of the following.
\begin{itemize}
\item A Parseval identity
\item A convolution identity
\item An inversion formula
\item A quasirandomness-measuring norm
\item An inverse theorem for the quasirandomness-measuring norm
\end{itemize}
Sometimes we can indeed get all of these, but in situations where we can't, it turns out that just having the last two properties is sufficient for some very interesting applications. In several of these situations, it remains a fascinating challenge to find new versions of the generalizations with improved properties.

\bibliographystyle{siam}
\bibliography{myreferences}

\end{document}